\documentclass{article}
\usepackage{amsmath}
\usepackage{amsfonts}
\usepackage{amssymb}
\usepackage{graphicx}
\usepackage{fancyhdr}%
\usepackage{xcolor}%
\usepackage[titletoc,title]{appendix}
\usepackage{combelow}
\usepackage{url}

\providecommand{\U}[1]{\protect\rule{.1in}{.1in}}
\numberwithin{equation}{section}
\newtheorem{theorem}{Theorem}[section]

\newtheorem{definition}[theorem]{Definition}

\newtheorem{lemma}[theorem]{Lemma}

\newtheorem{proposition}[theorem]{Proposition}
\newtheorem{remark}[theorem]{Remark}

\newenvironment{proof}[1][Proof]{\noindent\textbf{#1.} }{\ \rule{0.5em}{0.5em}}
\begin{document}

\title{On the Trace of $\dot W_{a}^{m+1,1}(\mathbb{R}_{+}^{n+1})$}
\author{Giovanni Leoni\\Department of Mathematical Sciences, \\Carnegie Mellon University, \\Pittsburgh PA 15213-3890, USA
\and Daniel Spector\\Department of Mathematics,\\National Taiwan Normal University, \\Taipei City, Taiwan 116, R.O.C.\\
and \\
National Center for Theoretical Sciences\\No. 1 Sec. 4 Roosevelt Rd., National Taiwan
University\\Taipei, 106, Taiwan}
\maketitle

\begin{abstract}
In this paper we prove extension results for functions in Besov spaces.  Our results are new in the homogeneous setting, while our technique applies equally in the inhomogeneous setting to obtain new proofs of classical results.  While our results include $p>1$, of principle interest is the case $p=1$, where we show that
\begin{equation*}
\int_{\mathbb{R}_{+}^{n+1}}t^{a}|\nabla^{m+1}u(x,t)|\;dtdx\lesssim\left\vert
f\right\vert _{B^{m-a,1}(\mathbb{R}^{n})}
\end{equation*}
for all $f \in \dot{B}^{m-a,1}(\mathbb{R}^{n})$ (the homogeneous Besov space) where $u$ is a suitably scaled heat extension of $f$.
\end{abstract}

\section{Introduction}\label{intro}

In the classical paper \cite{gagliardo1957}, Gagliardo proved that when
$1<p<\infty$, the trace space of $W^{1,p}(\Omega)$ is the fractional Sobolev
space $W^{1-1/p,p}(\partial\Omega)$ (see \cite{leoni-book-fractional}). Here,
$\Omega$ is an open bounded set of $\mathbb{R}^{n+1}$ with smooth boundary. An induction argument gives that the trace space of
$W^{m+1,p}(\Omega)$ is $W^{m+1-1/p,p}(\partial\Omega)$ for $m\in\mathbb{N}$
and $1<p<\infty$.

The history is somewhat more involved when $p=1$. In the first order case, Gagliardo proved in the same paper that the trace space of $W^{1,1}(\Omega)$ is $L^{1}(\partial\Omega)$ (see also \cite{mironescu2015} or \cite[Theorem 18.13]{leoni-book-sobolev} for a simpler proof due to Mironescu).  However, in the higher order case the trace of $W^{m+1,1}(\Omega)$ for
$m\in\mathbb{N}$ is not $W^{m,1}(\partial\Omega)$. Indeed, Uspenski\u{\i}
\cite{uspenskii1970} considered the homogeneous weighted Sobolev spaces
$\dot{W}_{a}^{m+1,p}(\mathbb{R}_{+}^{n+1})$, defined as the space of all
functions $u\in W_{\operatorname*{loc}}^{m+1,p}(\mathbb{R}_{+}^{n+1})$ such
that%
\[
|u|_{W_{a}^{m+1,p}(\mathbb{R}_{+}^{n+1})}:=\left(
\int_{\mathbb{R}^{n+1}_{+}}t^{a}|\nabla^{m+1}u(x,t)|^{p}dxdt\right)  ^{1/p}<\infty,
\]
where 
$$(x,t)\in\mathbb{R}^{n}\times (0,\infty)=:\mathbb{R}^{n+1}_{+},$$
$m\in\mathbb{N}_{0}$, $a>-1$, and $1<p<\infty$, and where we use
$\nabla^{m+1}u:=\nabla^{m}(\nabla u)$ to denote the inductively defined higher
order gradient. He proved that when $a-p(m+1)+1<0$, the trace space of the \emph{inhomogeneous} Sobolev space $\dot
{W}_{a}^{m+1,p}(\mathbb{R}_{+}^{n+1})\cap L^{p}(\mathbb{R}_{+}^{n+1})$ is
given by the Besov space $B^{m+1-(a+1)/p,p}(\mathbb{R}^{n})$, that is,
\begin{equation}
\operatorname*{Tr}(\dot{W}_{a}^{m+1,p}(\mathbb{R}_{+}^{n+1})\cap
L^{p}(\mathbb{R}_{+}^{n+1}))=B^{m+1-(a+1)/p,p}(\mathbb{R}^{n}%
).\label{trace space}%
\end{equation} 


As noted in the literature (see, e.g., \cite[p.295]{burenkov-book1998}), \cite[p.515]{mazya-book2011}, \cite{mironescu-russ-2015}), while Uspenski\u{\i}'s result is only stated for $1<p<\infty$, his proof extends\footnote{The reader should note here there is a slight inaccuracy in the assertion of how to handle the estimate for $r-1$ odd in the case division on p.~137, but ultimately this case is not needed for the demonstration of his Theorem 3.} without modification to the case $p=1$.  In particular, Uspenski\u{\i}'s trace theorem \cite[Theorem 2]{uspenskii1970} gives 
\begin{equation}
\operatorname*{Tr}(\dot{W}_{a}^{m+1,1}(\mathbb{R}_{+}^{n+1})\cap
L^{1}(\mathbb{R}_{+}^{n+1}))\subseteq
B^{m-a,1}(\mathbb{R}^{n})\label{inclusion easy}%
\end{equation}
(see also \cite{demengel-1984}, \cite[Theorem 18.57]{leoni-book-sobolev} and
\cite{mironescu-russ-2015}), while his lifting theorem \cite[Theorem 3]{uspenskii1970} shows
\begin{equation}
B^{m-a,1}(\mathbb{R}^{n}) \subseteq \operatorname*{Tr}(\dot{W}_{a}^{m+1,1}(\mathbb{R}_{+}^{n+1})\cap
L^{1}(\mathbb{R}_{+}^{n+1})).
\label{inclusion hard}%
\end{equation}
Both directions of Uspenski\u{\i}'s argument are a little tricky, though their presentation has been streamlined by Maz'ya in \cite[Theorem 1 in Section 10.1]{mazya-book2011}.  As was observed by Mironescu and Russ \cite{mironescu-russ-2015}, the lifting argument in \cite{mazya-book2011} is missing the estimate for the cross term (second order mixed derivatives in the trace and normal variable). 
 This is a natural motivation for their work \cite{mironescu-russ-2015}, where utilizing Littlewood--Paley theory, they give a
simple proof of the equality \eqref{trace space} that includes the
case $p=1$, see \cite[Theorems 1.1 and 1.2]{mironescu-russ-2015}.
Accordingly, their paper makes use of the Littlewood--Paley characterization
of Besov spaces (see \cite[Theorem 17.77]{leoni-book-sobolev}). 

In this paper, we are interested in the question of extension of functions in \emph{homogeneous} Besov spaces, which arise naturally in the study of PDE on unbounded domains.  Our approach is simple and has the benefit of being clearly well-defined in both the inhomogeneous and homogeneous setting.  The main new idea in this paper is to replace Uspenski\u{\i}'s use of the harmonic extension of $f$ with
\begin{equation}
u(x,t):=(W_{t}\ast f)(x)=\int_{\mathbb{R}^{n}}W_{t}(x-y)f(y)\,dy,
\label{function u}%
\end{equation}
where $W$ is the Gaussian function%
\begin{equation}
W(x):=\frac{\exp(-|x|^{2}/4)}{(4\pi)^{n/2}},\quad W_{t}(x):=\frac{1}{t^{n}%
}W(xt^{-1})=\frac{\exp(-|x|^{2}/(4t^2))}{(4\pi t^{2})^{n/2}}.
\label{gaussian}%
\end{equation}
While this rescaled Gauss-Weierstrauss extension has been utilized previously in the literature (see, e.g. Taibleson \cite[p.~458]{taibleson1964}), it is by no means obvious that it gives a relatively simple resolution of the extension question.  A first clear gain is that the convolution \eqref{function u} is always well-defined for functions $f \in \dot{B}^{m-a,1}(\mathbb{R}^{n})$, in contrast to the harmonic extension utilized by Uspenski\u{\i}.  The main point of interest is a simplification in the estimates which comes from a different underlying PDE.  In particular, as $p(x,t)=W_{\sqrt{t}}(x)$ is the heat kernel, $\frac{\partial p}{\partial t}=\Delta p$. Hence, using the chain rule, or a
direct computation, we have that
\begin{equation}
\frac{\partial W_{t}}{\partial t}=2t\Delta W_{t}, \label{heat equation}%
\end{equation}
so that $u$ satisfies the degenerate parabolic initial value problem%
\[
\left\{
\begin{array}
[c]{cc}%
\frac{\partial u}{\partial t}=2t\Delta u & \text{in }\mathbb{R}_{+}^{n+1},\\
u(x,0)=f(x) & \text{in }\mathbb{R}^{n}.
\end{array}
\right.
\]
We will see shortly the usefulness of this relation.  We first state the main result of this paper in
\begin{theorem}
\label{theorem higher order}Let $m\in\mathbb{N}_{0}$ and $-1<a<m$. Suppose
that $f\in \dot{B}^{m-a,1}(\mathbb{R}^{n})$ and let $u$ be given by
\eqref{function u}. Then
\begin{equation}
\int_{\mathbb{R}_{+}^{n+1}}t^{a}|\nabla^{m+1}u(x,t)|\;dtdx\lesssim\left\vert
f\right\vert _{B^{m-a,1}(\mathbb{R}^{n})}.\label{main estimate}%
\end{equation}
\end{theorem} 
Note that when $a=0$ and $p=1$, this provides a lifting for $\dot{B}^{m,1}(\mathbb{R}^{n})$ into $\dot{W
}^{m+1,1}(\mathbb{R}^{n+1}_{+})$.

When $a$ is an integer, the idea of the proof of Theorem
\ref{theorem higher order} is based on Uspenski\u{\i}'s trick of introducing
second order differences and his use of repeated harmonic extension.   In particular, replacing the Poisson kernel with the Gauss-Weierstrauss kernel, Uspenski\u{\i}'s ansatz \cite[p.~137-138]{uspenskii1970}, reads 
\begin{align}
\frac{\partial^{2}u}{\partial x_{i}\partial x_{j}}(x,t)=\frac{1}{2}%
t^{-n-2}\int_{\mathbb{R}^{n}}\frac{\partial^{2}W_1}{\partial x_{i}\partial
x_{j}}(ht^{-1})[f(x+h)+f(x-h)-2f(x)]\,dh,\label{usp-prime}
\end{align}
for $i,j=1,\ldots, n$.  This relies on the fact that $\frac{\partial
^{2}W_1}{\partial x_{i}\partial x_j}$ is even and has mean zero, and allows one to introduce the appropriate quantity on the right hand side when one has exactly two pure second order derivatives.  The point is then that estimates for the entries of the tensor
\[
t^{a}\nabla^{m+1}u
\]
can be reduced to this case through the use of the identity \eqref{heat equation} and the semi-group property of the Gaussian.  The former allows one to directly trade derivatives in $t$ for derivatives in the trace variable, up to a polynomial in the normal variable that is harmless, which in combination with the latter reduces the question of estimate for $t^{a}\nabla^{m+1}u$ to estimates for linear combinations of
\[
t^{a+l}\frac{\partial^{\gamma^{\prime}}W_{t/\sqrt{2}}}{\partial x^{\gamma^{\prime}}%
}\ast\frac{\partial^{2}W_{t/\sqrt{2}}}{\partial x_{i}\partial x_{j}}\ast
\frac{\partial^{\alpha}f}{\partial x^{\alpha}}%
\]
for some $l\in\{0,\ldots,m+1\}$, and where the multi-indices $\alpha,\gamma
\in\mathbb{N}_{0}^{n}$ satisfy $|\alpha|=m-k-1$ and $|\gamma^{\prime}|=l-k$, with $k$ the integer part of $m-a$.  The two pure second order derivatives in the trace variable are then amenable to the analog of Uspenski\u{\i}'s ansatz \eqref{usp-prime}, and the estimate follows, where one uses rapid decay of the Gaussian to ensure convergence of several rescaled integrals in the estimates.

In the proof of Theorem \ref{theorem higher order}, we will show that a function $f\in \dot{B}^{m-a,1}(\mathbb{R}^{n})$ has at most polynomial growth, so that in particular \eqref{function u} is well-defined.  This is no longer the case if we replace $W_t$ with the Poisson kernel, as in Uspenski\u{\i}'s paper.  This is not an issue for the extensions utilized by Mironescu and Russ \cite[Equation (4.1) on p.~362]{mironescu-russ-2015}, however, to use their work, in addition to subtraction of a suitable polynomial for the applicability of the fundamental theorem of calculus in Lemma 4.1 on p.~362, one should establish a density result for $C^\infty_c(\mathbb{R}^n)$ in the homogeneous spaces, a question which is itself non-trivial, see e.g. \cite[p.~Theorem 6.107 on p.~251]{leoni-book-fractional}.  


One has an analog of Theorem \ref{theorem higher order} for $1<p<\infty$ by a similar argument.  
%

\begin{theorem}
\label{theorem p}Let $m\in\mathbb{N}_{0}$, $1\leq p<\infty$, and
$-1<a<p(m+1)-1$. Suppose that $f\in \dot{B}^{m+1-(a+1)/p,p}(\mathbb{R}^{n})$
and let $u$ be given by \eqref{function u}. Then
\[
\int_{\mathbb{R}_{+}^{n+1}}t^{a}|\nabla^{m+1}u(x,t)|^{p}dxdt\lesssim\left\vert
f\right\vert _{B^{m+1-(a+1)/p,p}(\mathbb{R}^{n})}^{p}.
\]

\end{theorem}

Next, we turn our attention to the inhomogeneous case. For $1<p<\infty$, Triebel in \cite[Theorem 2.9.1 on p. 214]{triebel-book1995} considered the (inhomogeneous) weighted
Sobolev space $W_{a}^{m+1,p}(\mathbb{R}_{+}^{n+1})$ defined as the space of
all functions $u\in W_{\operatorname*{loc}}^{m+1,p}(\mathbb{R}_{+}^{n+1})$
such that
\[
\Vert u\Vert_{W_{a}^{m+1,p}(\mathbb{R}_{+}^{n+1})}:=\left(  \int_{\mathbb{R}_{+}^{n+1}}t^{a}|u(x,t)|^{p}dxdt\right)  ^{1/p}+\sum_{j=1}%
^{m+1}|u|_{W_{a}^{j,p}(\mathbb{R}_{+}^{n+1})}<\infty.
\]
He proved that for $-1/p<a<m+1-1/p$, the mapping%
\[
u\mapsto\left(  u(x,0),\frac{\partial u}{\partial t}(x,0),\ldots
,\frac{\partial^{l}u}{\partial t^{l}}(x,0)\right)
\]
is a retraction from $W_{a}^{m+1,p}(\mathbb{R}_{+}^{n+1})$ onto $\prod
_{j=0}^{l}B^{m+1-j-a-1/p}(\mathbb{R}^{n})$. Here $l=\left\lceil
m-a+1/p\right\rceil $, where $\left\lceil s\right\rceil $ is the floor of $s$. The lifting makes use of the harmonic extension \eqref{function F harmonic}.

Triebel also showed that if $a\geq m+1-1/p$, then $$W_{a}^{m+1,p}%
(\mathbb{R}_{+}^{n+1})=W_{0,a}^{m+1,p}(\mathbb{R}_{+}^{n+1}),$$ where
$W_{0,a}^{m+1,p}(\mathbb{R}_{+}^{n+1})$ is the completion of $C_{c}^{\infty
}(\mathbb{R}_{+}^{n+1})$ with respect to the norm in $W_{a}^{m+1,p}%
(\mathbb{R}_{+}^{n+1})$. In particular, this implies that, in this case, the
trace operator cannot be continuous since we can approximate a smooth function
in $W_{a}^{m+1,p}(\mathbb{R}_{+}^{n+1})$ with non zero trace with a sequence
of functions in $C_{c}^{\infty}(\mathbb{R}_{+}^{n+1})$. See also the paper
\cite{grisvard1963} of Grisvard for the case $m=0$. 

We carry out this program in case $p=1$ in the following three theorems.

\begin{theorem}
\label{theorem non-homogeneous} Let $m\in\mathbb{N}_{0}$ and $-1<a\leq m$. If
$a<m$, then for every $f\in B^{m-a,1}(\mathbb{R}^{n})$, there exists $F\in
W_{a}^{m+1,1}(\mathbb{R}_{+}^{n+1})$ such that $\operatorname*{Tr}(F)=f$ and%
\[
\Vert F\Vert_{W_{a}^{m+1,1}(\mathbb{R}_{+}^{n+1})}\lesssim\Vert f\Vert
_{B^{m-a,1}(\mathbb{R}^{n})}.
\]
On the other hand, if $a=m$, then for every $f\in L^{1}(\mathbb{R}^{n})$ there
exists $F\in W_{m}^{m+1,1}(\mathbb{R}_{+}^{n+1})$ such that
$\operatorname*{Tr}(F)=f$ and%
\begin{equation}
\Vert F\Vert_{W_{m}^{m+1,1}(\mathbb{R}_{+}^{n+1})}\lesssim\Vert f\Vert
_{L^{1}(\mathbb{R}^{n})}.
 \label{z1}   
\end{equation}

\end{theorem}
The lifting in \eqref{z1} was obtained by Mironescu and Russ \cite[Proposition 1.14]{mironescu-russ-2015}.

The
following result is critical in reducing elliptic or parabolic boundary
value problems with inhomogeneous boundary conditions to homogeneous
ones (see, e.g., \cite[Theorem 4.2.2 on p.218]{necas-book2012} or \cite{mazya-mitrea-shaposhnikova-2010} for the case $p>1$). See also the recent work of
Gmeineder, Raita, and Van Schaftingen
\cite{gmeineder-raita-van-schaftingen2022} for an application to boundary
ellipticity. 

\begin{theorem}
\label{theorem trace higher}Let $m\in\mathbb{N}$ and $-1<a<m$. If
$a=k\in\mathbb{N}_{0}$, suppose that $f_{j}\in B^{m-k-j,1}(\mathbb{R}^{n})$ for
$j=0,\ldots,m-k-1$, and $f_{m-k}\in L^{1}(\mathbb{R}^{n})$. Then there exists
$F\in W_{a}^{m+1,1}(\mathbb{R}_{+}^{n+1})$ such that $\operatorname*{Tr}%
(\,F)=f_{0}$, $\operatorname*{Tr}(\,\frac{\partial^{j}F}{\partial t^{j}%
})=f_{j}$ for $j=1,\ldots,m-k-1$, and%
\begin{equation}
\Vert F\Vert_{W_{k}^{m+1,1}(\mathbb{R}_{+}^{n+1})}\lesssim\sum_{j=0}%
^{m-k-1}\Vert f_{j}\Vert_{B^{m-k-j,1}(\mathbb{R}^{n})}+\Vert f_{m-k}%
\Vert_{L^{1}(\mathbb{R}^{n})}.\label{trace estimate higher}%
\end{equation}
On the other hand, if $a\notin\mathbb{N}_{0}$, suppose that $f_{j}\in
B^{m-a-j,1}(\mathbb{R}^{n})$ for $j=0,\ldots,l$, where $l:=\lfloor m-a\rfloor
$. Then there exists $F\in W_{a}^{m+1,1}(\mathbb{R}_{+}^{n+1})$ such that
$\operatorname*{Tr}(\,F)=f_{0}$, $\operatorname*{Tr}(\,\frac{\partial^{j}%
F}{\partial t^{j}})=f_{j}$ for $j=1,\ldots,l$, and%
\[
\Vert F\Vert_{W_{a}^{m+1,1}(\mathbb{R}_{+}^{n+1})}\lesssim\sum_{j=0}^{l}\Vert
f_{j}\Vert_{B^{m-a-j,1}(\mathbb{R}^{n})}.
\]

\end{theorem}

Finally, we discuss the case $a>m$. 

\begin{theorem}
\label{theorem a>m}Let $m\in\mathbb{N}_{0}$ and $a>m$. Then $W_{a}^{m+1,1}(\mathbb{R}_{+}^{n+1}%
)=W_{0,a}^{m+1,1}(\mathbb{R}_{+}^{n+1})$. 
\end{theorem}

This paper is organized as follows. In Section \ref{section preliminaries}, we discuss some basic properties of Besov spaces. 
In Section \ref{section traces}, 
we prove Theorems \ref{theorem higher order} and \ref{theorem p}. Section \ref{section non-homogeneous} deals with the inhomogeneous case: We prove Theorems \ref{theorem non-homogeneous}, \ref{theorem trace higher}, and \ref{theorem a>m}.  Finally, in Section \ref{section poisson}, we prove several extension results via harmonic extension.  Here we also show how Uspenski\u{\i}'s characterization of the trace of $W^{2,1}(\mathbb{R}^{n+1}_+)$ by harmonic extension yields a simple proof of the boundedness of the Riesz transforms on $B^{1,1}(\mathbb{R}^n)$, thus giving a short proof of the latter fact which is a standard consequence of Littlewood--Paley theory  (see, e.g., \cite{sawano2020} or \cite[Section
5.2.2]{triebel-book2010}).

\section{Preliminaries}
\label{section preliminaries}
In this section, we present some basic properties of Besov spaces and the Riesz
transform that we will use in the sequel. Throughout this paper, the
expression
\begin{equation}
\mathcal{A}\lesssim\mathcal{B}\text{\quad means }\mathcal{A}\leq C\mathcal{B}%
\nonumber
\end{equation}
for some constant $C>0$ that depends on the parameters quantified in the
statement of the result (usually $n$ and $p$), but not on the functions and
their domain of integration.

Given $a\in\mathbb{R}$ and $1\le p<\infty$, we denote by $L^p_a(\mathbb{R}^{n+1}_+)$
the space all measurable functions $f:\mathbb{R}^{n+1}_+\rightarrow\mathbb{R}$ such that
\begin{equation}
    \Vert f\Vert_{L^p_a(\mathbb{R}^{n+1}_+)}:=\left(\int_{\mathbb{R}^{n+1}_+}t^a|f(x,t)|^pdxdt\right)^{1/p}
    <\infty.\label{Lp weighted}
\end{equation}

Given a function $f\in W_{\operatorname*{loc}%
}^{m,p}(\Omega)$, where $\Omega\subseteq\mathbb{R}^{n+1}$, it will be convenient to have several different symbols to denote various derivatives beyond what we have introduced in Section \ref{intro}, for $k=1,\ldots, m$,
\begin{itemize}
    \item $\nabla^k f$ the inductively defined gradient jointly in $(x,t)$ of order $k$;
    \item $\frac{\partial^\alpha f}{\partial x^\alpha}$ the $k$th order partial derivative of $f$ in $x$ given by the multi-index $\alpha \in \mathbb{N}_0^n$ with $|\alpha|=k$.
\end{itemize}

  In particular, for $k=1,\ldots, m$, we also denote by
\begin{itemize}
    \item $\nabla_x^k f$ the inductively defined gradient in $x$ of order $k$;
    \item $\frac{\partial f}{\partial x_i}$ the first order partial derivative of $f$ with respect to the trace variable $x_i$; 
    \item $\frac{\partial^k f}{\partial t^k} $ the $k$th order partial derivative of $f$ with respect to the extension variable;
    \item $\partial^\alpha f$ the partial derivative of $f$ in $(x,t)$ given by the multi-index $\alpha \in \mathbb{N}_0^k\times \mathbb{N}$ with $|\alpha|=k$.
\end{itemize}
 
\begin{definition}
Given an open set $\Omega\subseteq\mathbb{R}^{n+1}$, $m\in\mathbb{N}$, and
$1\leq p\leq\infty$, we say that a function $f\in W_{\operatorname*{loc}%
}^{m,p}(\Omega)$ belongs to the \emph{homogeneous Sobolev space }$\dot{W}%
^{m,p}(\Omega)$ if $|\nabla^m f| \in L^{p}(\Omega)$.
\end{definition}

Given a function $f:\mathbb{R}^{n}\rightarrow\mathbb{R}$ and $x,h\in
\mathbb{R}^{n}$, we write%
\begin{equation}
\Delta_{h}f(x):=f(x+h)-f(x),\quad\Delta_{h}^{k+1}f(x):=\Delta_{h}^k(\Delta
_{h}f(x)). \label{difference}%
\end{equation}
Observe that 
$$ \Delta_{h}^{2}f(x)=f(x+2h)-2f(x+h)+f(x)$$
and
$$ \Delta_{h}^{2}f(x-h)=f(x+h)-2f(x)+f(x-h).$$

\begin{definition}
Given $1\le p,q<\infty$ and $0<s\le1$, we say that a function $f\in
L_{\operatorname*{loc}}^{p}(\mathbb{R}^{n})$ belongs to the \emph{homogeneous
Besov space} $\dot{B}^{s,p}_{q}(\mathbb{R}^{n})$ if%
\[
|f|_{B^{s,p}_{q}(\mathbb{R}^{n})}:=\left(  \int_{\mathbb{R}^{n}}\|
\Delta^{\lfloor s\rfloor+1}_{h}f\|_{L^{p}(\mathbb{R}^{n})}^{q} \frac
{dh}{|h|^{n+sq}}\right)  ^{1/q}<\infty,
\]
where $\lfloor s\rfloor$ is the integer part of $s$. The (inhomogeneous)
\emph{Besov space} $B^{s,p}_{q}(\mathbb{R}^{n})$ is the space of all functions
$f\in L^{p}(\mathbb{R}^{n})\cap\dot{B}^{s,p}_{q}(\mathbb{R}^{n})$ endowed with
norm%
\[
\Vert f\Vert_{B^{s,p}_{q}(\mathbb{R}^{n})}:=\Vert f\Vert_{L^{p}(\mathbb{R}%
^{n})}+|f|_{B^{s,p}_{q}(\mathbb{R}^{n})}.
\]

\end{definition}

When $q=p$, we write $\dot{B}^{s,p}(\mathbb{R}^{n})$ and $B^{s,p}%
(\mathbb{R}^{n})$ for $\dot{B}^{s,p}_{p}(\mathbb{R}^{n})$ and $B^{s,p}%
_{p}(\mathbb{R}^{n})$, respectively.

In what follows, we will use the equivalent seminorm for $\dot{B}%
^{1,1}(\mathbb{R}^{n})$:%
\[
|f|_{B^{1,1}(\mathbb{R}^{n})}^{\infty}:=\int_{0}^{\infty}\sup_{|h|\leq r}\|
\Delta^{2}_{h}f\|_{L^{1}(\mathbb{R}^{n})}\frac{dr}{r^{2}}%
\]
(see \cite[Proposition 17.17]{leoni-book-sobolev}).

\begin{definition}
Given $1\le p,q<\infty$ and $s>1$, we define the \emph{homogeneous Besov
space} $\dot{B}^{s,p}_{q}(\mathbb{R}^{n})$ as the space of all functions $f\in
W_{\operatorname*{loc}}^{\ell,p}(\mathbb{R}^{n})$ such that $\frac
{\partial^{\alpha}f}{\partial x^{\alpha}}\in\dot{B}^{s-\ell,p}_{q}%
(\mathbb{R}^{m})$ for all multi-indices $\alpha\in\mathbb{N}_{0}^{n}$ with
$|\alpha|=\ell$, where $\ell=\max\{m\in\mathbb{N}:\,m<s\}$. The
(inhomogeneous) \emph{Besov space} $B^{m,p}_{q}(\mathbb{R}^{n})$ is the
space of all functions $f\in W^{\ell,p}(\mathbb{R}^{n})\cap\dot{B}^{s,p}%
_{q}(\mathbb{R}^{n})$ endowed with norm%
\[
\Vert f\Vert_{B^{s,p}_{q}(\mathbb{R}^{n})}:=\Vert f\Vert_{W^{\ell
,p}(\mathbb{R}^{n})}+\sum_{|\alpha|=m-\ell}\left|  \frac{\partial^{\alpha}%
f}{\partial x^{\alpha}}\right|  _{B^{s-\ell,p}_{q}(\mathbb{R}^{n})}.
\]

\end{definition}

Note that $\ell=\lfloor s\rfloor$ if $s\notin\mathbb{N}$ and $\ell=\lfloor
s\rfloor-1$ if $s\in\mathbb{N}$. As before, when $q=p$, we write $\dot
{B}^{s,p}(\mathbb{R}^{n})$ and $B^{s,p}(\mathbb{R}^{n})$ for $\dot{B}%
^{s,p}_{p}(\mathbb{R}^{n})$ and $B^{s,p}_{p}(\mathbb{R}^{n})$, respectively.

It is important to remark that when $s\notin \mathbb{N}$, the Besov spaces $B^{s,p}(\mathbb{R}^n)$ coincide with the fractional Sobolev spaces $W^{s,p}(\mathbb{R}^n)$, while for $s=k\in \mathbb{N}$,
$$B^{k,p}(\mathbb{R}^n)\subsetneq W^{k,p}(\mathbb{R}^n).$$
For the continuous embedding, we refer to \cite[Theorem 17.66]{leoni-book-sobolev}.  The two spaces are not equivalent.  In the case $p>1$ this follows from \cite[Theorems 2.3.9 and 2.5.6]{triebel1973}. 
 When $p=1$, there is a simple example for $k=1$:  Assume by contradiction that 
$$\|f\|_{B^{1,1}(\mathbb{R}^n)}\lesssim \|f\|_{W^{1,1}(\mathbb{R}^n)} $$
for all $f\in W^{1,1}(\mathbb{R}^n)$. It follows by a mollification argument that 
$$\|f\|_{B^{1,1}(\mathbb{R}^n)}\lesssim \|f\|_{L^1(\mathbb{R}^n)}+|Df|(\mathbb{R}^n) $$
for all $f\in BV(\mathbb{R}^n)$. 
Take $f = \chi_{[0,1]^n} \in BV(\mathbb{R}^n)$. 
Given $x\in\mathbb{R}^{n}$, we write
$$x=(x',x_n)\in \mathbb{R}^{n-1}\times \mathbb{R}.$$ Then, by the change of variables $h'=h_n z'$, we obtain
\begin{align*}
|f&|_{B^{1,1}(\mathbb{R}^{n})}
\\
&\geq   \int_{[0,1]^{n-1}\times [0,1/2]} \int_{[0,1]^{n-1}\times [x_n,1/2]}  \frac
{|\Delta^2_h\chi_{[0,1]^n}(x-h)|}{|h|^{n+1}} dh'dh_ndx'dx_n\\
&= \int_{[0,1/2]} \int_{[x_n,1/2]} \int_{[0,1]^{n-1}}  \frac
{dh'}{(|h'|^2+h_n^2)^{(n+1)/2}}dh_ndx_n\\
&\geq \int_{[0,2]^{n-1}}  \frac{dz'}{(|z'|^2+1)^{(n+1)/2}}   \int_{[0,1/2]} \int_{[x_n,1/2]}\frac{1}{h_n^2} dh_ndx_n \\
&=\infty.
\end{align*}
When $k>p=1$, this example can be modified as follows.  Let $\varphi \in C^\infty_c([-1/2,3/2]^n)$ be a function such that $\varphi\equiv 1$ on $[0,1]^n$.  Define
\begin{align*}
    \psi(x):= \varphi(x) \int_{-1}^{x_n} \int_{-1}^{s_{k-2}}\dots \int_{-1}^{s_2} \chi_{\{x_n>0\}}(s_1)\varphi(x',s_1)\;ds_1\ldots ds_{k-1}.
\end{align*}
We claim $D (\nabla^{k-1}\psi) \in M_b(\mathbb{R}^n;\mathbb{R}^{n\times k})$.  In fact, for every $\alpha \neq (0,\ldots,k)$ such that $|\alpha|=k$
\begin{align*}
\frac{\partial^\alpha \psi}{\partial x^\alpha} 
\end{align*}
is a bounded, compactly supported function, so that it only remains to observe that
\begin{align*}
\left(D\frac{\partial^{k-1} \psi}{\partial x_n^{k-1}}\right)_k(x) = \varphi^2(x)\mathcal{H}^{n-1}|_{\{x_n=0\}} + \tilde{\psi} 
\end{align*}
for some bounded compactly supported function $\tilde{\psi}$ and the claim follows.  On the other hand, $\psi \notin B^{k,1}(\mathbb{R}^n)$ since
\begin{align*}
\left\|\frac{\partial^{k-1} \psi}{\partial x_n^{k-1}}\right\|_{B^{1,1}(\mathbb{R}^n)} \geq \|\chi_{\{x_n>0\}}\varphi^2\|_{B^{1,1}(\mathbb{R}^n)} - \|\tilde{\psi}\|_{B^{1,1}(\mathbb{R}^n)},
\end{align*}
and while
\begin{align*}
\|\tilde{\psi}\|_{B^{1,1}(\mathbb{R}^n)} <\infty,
\end{align*}
a computation similar to the preceding shows
\begin{align*}
\|\chi_{\{x_n>0\}}\varphi^2\|_{B^{1,1}(\mathbb{R}^n)}=\infty.
\end{align*}

The following result is well known (see, e.g., \cite[Theorem 17.24]{leoni-book-sobolev} for a proof that uses abstract interpolation).

\begin{proposition}
\label{proposition embedding}Let $0<s\leq1$. Then $W^{\lfloor s\rfloor
+1,1}(\mathbb{R}^{n})$ is continuously embedded in $B^{s,1}(\mathbb{R}^{n})$.
\end{proposition}

\begin{proof}
Assume that $s=1$ and let $f\in W^{2,1}(\mathbb{R}^{n})$. By the fundamental
theorem of calculus,%
\[
\Vert\Delta_{h}^{2}f\Vert_{L^{1}(\mathbb{R}^{n})}\leq|h|^{2}\Vert\nabla
^{2}_xf\Vert_{L^{1}(\mathbb{R}^{n})},
\]
and so%
\begin{align*}
\int_{\mathbb{R}^{n}}\Vert\Delta_{h}^{2}f\Vert_{L^{1}(\mathbb{R}^{n})}%
\frac{dh}{|h|^{n+1}}  & =\int_{B(0,1)}\Vert\Delta_{h}^{2}f\Vert_{L^{1}%
(\mathbb{R}^{n})}\frac{dh}{|h|^{n+1}}+\int_{\mathbb{R}^{n}\setminus
B(0,1)}\Vert\Delta_{h}^{2}f\Vert_{L^{1}(\mathbb{R}^{n})}\frac{dh}{|h|^{n+1}%
}\\
& \leq\Vert\nabla^{2}_xf\Vert_{L^{1}(\mathbb{R}^{n})}\int_{B(0,1)}\frac
{dh}{|h|^{n-1}}+2^{2}\Vert f\Vert_{L^{1}(\mathbb{R}^{n})}\int_{\mathbb{R}%
^{n}\setminus B(0,1)}\frac{dh}{|h|^{n+1}}\\
& \lesssim\Vert\nabla^{2}_xf\Vert_{L^{1}(\mathbb{R}^{n})}+\Vert f\Vert
_{L^{1}(\mathbb{R}^{n})}.
\end{align*}
The case $0<s<1$ is similar. We omit the details.
\end{proof}

\section{The Homogeneous Case}

\label{section traces}
In this section, we prove Theorems \ref{theorem higher order} and \ref{theorem p}. For simplicity of exposition, we present a version of Theorem
\ref{theorem higher order} in the second order case.

\begin{theorem}
\label{theorem trace B11}Let $f\in B^{1,1}(\mathbb{R}^{n})$ and let $u$ be defined as in \eqref{function u}. Then 
\begin{equation}
\Vert \nabla^2 u\Vert_{L^{1}(\mathbb{R}_{+}^{n+1})}\lesssim|f|_{B^{1,1}%
(\mathbb{R}^{n})}. \label{trace estimate B11}%
\end{equation}

\end{theorem}

\begin{lemma}
\label{lemma heat integral}Let $\alpha\in\mathbb{N}_{0}^{n}$ and
$b\in\mathbb{R}$ be such that $n+|\alpha|-b-1>0$. Then for every
$x\in\mathbb{R}^{n}\setminus\{0\}$,%
\[
\int_{0}^{\infty}t^{b}\left\vert \frac{\partial^{\alpha}W_{t}}{\partial
x^{\alpha}}(x)\right\vert \,dt\lesssim\frac{1}{|x|^{n+|\alpha|-b-1}}%
\]
and%
\[
\int_{\mathbb{R}^{n}}\left\vert \frac{\partial^{\alpha}W_{t}}{\partial
x^{\alpha}}(x)\right\vert \,dx\lesssim\frac{1}{t^{|\alpha|}}.
\]

\end{lemma}

\begin{proof}
The change of variables $t=|x|r^{-1}$, $dt=-|x|r^{-2}dr$ yields%
\begin{align*}
\int_{0}^{\infty}t^{b}\left\vert \frac{\partial^{\alpha}W_{t}}{\partial
x^{\alpha}}(x)\right\vert \,dt &  =\int_{0}^{\infty}\frac{t^{b}}%
{t^{n+|\alpha|}}\left\vert \frac{\partial^{\alpha}W}{\partial x^{\alpha}%
}(xt^{-1})\right\vert \,dt\\
&  =\frac{1}{|x|^{n+|\alpha|-b-1}}\int_{0}^{\infty}r^{n+|\alpha|-b-2}%
\left\vert \frac{\partial^{\alpha}W}{\partial x^{\alpha}}(rx/|x|)\right\vert
\,dr.
\end{align*}
Since $n+|\alpha|-b-2>-1$ we have that $r^{n+|\alpha|-b-2}$ is integrable near
$0$, which together with the fact that $W$ is a Gaussian, gives convergence of
the integral on the right-hand side.

The facts that $\frac{\partial^{\alpha}W_{t}}{\partial x^{\alpha}}(x)=\frac
{1}{t^{n+|\alpha|}}\frac{\partial^{\alpha}W}{\partial x^{\alpha}}(xt^{-1})$,
that $W$ is a Gaussian, and the change of variables $y=xt^{-1}$, $dy=t^{-n}dx$ imply
\[
\int_{\mathbb{R}^{n}}\left\vert \frac{\partial^{\alpha}W_{t}}{\partial
x^{\alpha}}(x)\right\vert \,dx=\frac{1}{t^{|\alpha|}}\int_{\mathbb{R}^{n}%
}\left\vert \frac{\partial^{\alpha}W}{\partial x^{\alpha}}(y)\right\vert
\,dy\lesssim\frac{1}{t^{|\alpha|}}.
\]

\end{proof}

We turn to the proof of Theorem \ref{theorem trace B11}

\begin{proof}[Proof of Theorem \ref{theorem trace B11}] Given $f\in B^{1,1}(\mathbb{R}^{n})$, let $u:=W_t*f$, where $W_t$ is defined in \eqref{gaussian}. 
\textbf{Step 1: } In this step we estimate the $L^1$ norms of $\frac{\partial^{2}u}{\partial x_{i}\partial x_{j}}$ and $\frac{\partial^{2}u}{\partial t^2}$. For any $i,j=1,\ldots n$, one has
\[
\frac{\partial^{2}u}{\partial x_{i}\partial x_{j}}(x,t)=\int_{\mathbb{R}^{n}%
}\frac{\partial^{2}W_{t}}{\partial x_{i}\partial x_{j}}(x-h)f(h)\,dh=\int%
_{\mathbb{R}^{n}}\frac{\partial^{2}W_{t}}{\partial x_{i}\partial x_{j}%
}(h)f(x-h)\,dh.
\]
Making use of the fact that
\[
\frac{\partial^{2}W_{t}}{\partial x_{i}\partial x_{j}}(-h)=\frac{\partial
^{2}W_{t}}{\partial x_{i}\partial x_{j}}(h)
\]
(the second order partial derivatives purely in the trace variable of the Gaussian 
kernel, even mixed, are even functions), by a change of variables, one also
has
\[
\frac{\partial^{2}u}{\partial x_{i}\partial x_{j}}(x,t)=\int_{\mathbb{R}^{n}%
}\frac{\partial^{2}W_{t}}{\partial x_{i}\partial x_{j}}(h)f(x+h)\,dh,
\]
Since
\[
\int_{\mathbb{R}^{n}}\frac{\partial^{2}W_{t}}{\partial x_{i}\partial x_{j}%
}(h)f(x)\,dh=0,
\]
this means one can write
\begin{equation}
\frac{\partial^{2}u}{\partial x_{i}\partial x_{j}}(x,t)=\frac{1}{2}%
\int_{\mathbb{R}^{n}}\frac{\partial^{2}W_{t}}{\partial x_{i}\partial x_{j}%
}(h)[f(x+h)+f(x-h)-2f(x)]\,dh.\label{100}%
\end{equation}
In particular, we can estimate
\begin{align*}
\int_{\mathbb{R}_{+}^{n+1}} &  \left\vert \frac{\partial^{2}%
u}{\partial x_{i}\partial x_{j}}(x,t)\right\vert \,dxdt\\
&  \leq\frac{1}{2}\int_{0}^{\infty}\int_{\mathbb{R}^{n}}\int_{\mathbb{R}^{n}%
}\left\vert \frac{\partial^{2}W_{t}}{\partial x_{i}\partial x_{j}%
}(h)\right\vert |f(x+h)+f(x-h)-2f(x)|\,dhdxdt\\
&  =\frac{1}{2}\int_{\mathbb{R}^{n}}\int_{\mathbb{R}^{n}}\int_{0}^{\infty
}\left\vert \frac{\partial^{2}W_{t}}{\partial x_{i}\partial x_{j}%
}(h)\right\vert \,dt|f(x+h)+f(x-h)-2f(x)|\,dhdx.
\end{align*}
By Lemma \ref{lemma heat integral}, 
\begin{equation}
\int_{0}^{\infty}\left\vert \frac{\partial^{2}W_{t}}{\partial x_{i}\partial
x_{j}}(h)\right\vert \,dt\lesssim\frac{1}{|h|^{n+1}}.\label{102}%
\end{equation}
Therefore
\begin{align*}
\int_{\mathbb{R}_{+}^{n+1}} &  \left\vert \frac{\partial^{2}%
u}{\partial x_{i}\partial x_{j}}(x,t)\right\vert \,dxdt\\
&  \lesssim\int_{\mathbb{R}^{n}}\int_{\mathbb{R}^{n}}\frac
{|f(x+h)+f(x-h)-2f(x)|}{|h|^{n+1}}\,dhdx.
\end{align*}
Since $\frac{\partial^{2}W_{t}}{\partial t^{2}}$ is even and integrates to zero, reasoning as before, we can write
\[
\frac{\partial^{2}u}{\partial t^{2}}(x,t)=\frac{1}{2}\int_{\mathbb{R}^{n}%
}\frac{\partial^{2}W_{t}}{\partial t^{2}}(h)[f(x+h)+f(x-h)-2f(x)]\,dh.
\]
Since $\frac{\partial^{2}W_{t}}{\partial t^{2}}$ can be written as a linear combination of 
$t^{-2}W_t$, $t^{-1}(x_it^{-1})\frac{\partial W_{t}}{\partial x_i}$, and 
$(x_it^{-1})(x_jt^{-1})\frac{\partial^{2}W_{t}}{\partial x_ix_j}$, reasoning as in the proof of Lemma \ref{lemma heat integral}, we have that
\begin{equation}
\int_{0}^{\infty}\left\vert \frac{\partial^{2}W_{t}}{\partial t^2}(h)\right\vert \,dt\lesssim\frac{1}{|h|^{n+1}}.\label{102a}%
\end{equation}
We can now continue as before to obtain the estimate for this derivative.

\textbf{Step 2: } In this step we estimate the $L^1$ norm of $\frac{\partial^{2}u}{\partial t\partial x_{j}}$. 
For the mixed derivatives involving $t$, one computes
\[
\frac{\partial^{2}u}{\partial t\partial x_{j}}(x,t)=\int_{\mathbb{R}^{n}}%
\frac{\partial W_{t}}{\partial t}(h)\frac{\partial f}{\partial x_{j}%
}(x-h)\,dh.
\]
For fixed $x$, define $g(h):=f(x-h)$. Then
\[
\frac{\partial f}{\partial x_{j}}(x-h)=-\frac{\partial g}{\partial h_{j}}(h),
\]
an integration by parts yields
\[
\frac{\partial^{2}u}{\partial t\partial x_{j}}(x,t)=\int_{\mathbb{R}^{n}}%
\frac{\partial^{2}W_{t}}{\partial t\partial x_{j}}(h)f(x-h)\,dh.
\]
Here, $\frac{\partial^{2}W_{t}}{\partial t\partial x_{j}}$ is not an even
function (and, in fact, it is odd). 
Since
\[
\frac{\partial W_{t}}{\partial t}=2t\Delta W_{t},
\]
we can write%
\[
\frac{\partial^{2}W_{t}}{\partial t\partial x_{i}}=2t\sum_{j=1}^{n}%
\frac{\partial^{3}W_{t}}{\partial x_{i}\partial^{2}x_{j}}.
\]
The semi-group property of the heat extension, which is just a manipulation of
the Fourier transform (see (\ref{gaussian})), leads to the identity\footnote{We have
\[
\mathcal{F}(W_{t})(\xi)=e^{-4t^{2}\pi^{2}|\xi|^{2}}=e^{-4(t/\sqrt{2})^{2}%
\pi^{2}|\xi|^{2}}e^{-4(t/\sqrt{2})^{2}\pi^{2}|\xi|^{2}}=\mathcal{F}%
(W_{t/\sqrt{2}})\mathcal{F}(W_{t/\sqrt{2}}).
\]
}%
\begin{equation}\label{semi-group}
\frac{\partial^{3}W_{t}}{\partial x_{i}\partial^{2}x_{j}}=\frac{\partial
W_{t/\sqrt{2}}}{\partial x_{i}}\ast\frac{\partial^{2}W_{t/\sqrt{2}}}{\partial
x_{j}^{2}}.
\end{equation}
Since $\frac{\partial^{2}W_{t/\sqrt{2}}}{\partial x_{j}^{2}}$ is even and has
zero average in the trace variable $x$, we can write%
\begin{equation}
\left(  \frac{\partial^{2}W_{t/\sqrt{2}}}{\partial x_{j}^{2}}\ast f\right)
(x)=\frac{1}{2}\int_{\mathbb{R}^{n}}\frac{\partial^{2}W_{t/\sqrt{2}}}{\partial
x_{j}^{2}}(h)\Delta_{h}^{2}f(x-h)\,dh,\label{c4}%
\end{equation}
where $\Delta^2_hf$ is as defined in \eqref{difference}.  In turn,%
\begin{equation}\label{u1}
\left(  \frac{\partial W_{t/\sqrt{2}}}{\partial x_{i}}\ast\frac{\partial
^{2}W_{t/\sqrt{2}}}{\partial x_{j}^{2}}\ast f\right)  (x)=\frac{1}{2}%
\int_{\mathbb{R}^{n}}\frac{\partial W_{t/\sqrt{2}}}{\partial x_{i}}%
(x-y)\int_{\mathbb{R}^{n}}\frac{\partial^{2}W_{t/\sqrt{2}}}{\partial x_{j}%
^{2}}(h)\Delta_{h}^{2}f(y-h)\,dhdy.
\end{equation}
In conclusion, we have%
\[
\frac{\partial^{2}u}{\partial t\partial x_{i}}(x,t)=\sum_{j=1}^{n}%
\int_{\mathbb{R}^{n}}\frac{\partial W_{t/\sqrt{2}}}{\partial x_{i}}%
(x-y)\int_{\mathbb{R}^{n}}t\frac{\partial^{2}W_{t/\sqrt{2}}}{\partial
x_{j}^{2}}(h)\Delta_{h}^{2}f(y-h)\,dhdy.
\]
Integrating this quantity over $\mathbb{R}_{+}^{n+1}$ and using Fubini's
theorem yields%
\begin{align*}
\int_{\mathbb{R}_{+}^{n+1}}&\left\vert \frac{\partial^{2}u}{\partial t\partial
x_{i}}(x,t)\right\vert \,dxdt\\
& \lesssim\sum_{j=1}^{n}\int_{\mathbb{R}^{n}}\left\vert \frac{\partial
W_{t/\sqrt{2}}}{\partial x_{i}}(z)\right\vert \,dz\int_{\mathbb{R}^{n}}%
\int_{\mathbb{R}^{n}}\int_{0}^{\infty}t\left\vert \frac{\partial^{2}%
W_{t/\sqrt{2}}}{\partial x_{j}^{2}}(h)\right\vert \,dt\left\vert \Delta
_{h}^{2}f(y-h)\right\vert \,dhdy\\
& \lesssim\int_{\mathbb{R}^{n}}\int_{\mathbb{R}^{n}}
\frac{\left\vert \Delta_{h}^{2}f(y-h)\right\vert}{|h|^{n+1}} \,dhdy\\
&=\int_{\mathbb{R}^{n}}\int_{\mathbb{R}^{n}}
\frac{\left\vert f(y+h)-2f(y)+f(y-h)\right\vert}{|h|^{n+1}} \,dhdy,
\end{align*}
where in the last inequality we applied twice Lemma \ref{lemma heat integral}.
\end{proof}

\begin{remark}\label{remark pure derivatives}
The proof in Step 1 is classical and follows Uspenski\u{\i}'s ansatz \eqref{usp}. Note that in Step 1, we only used the fact that the kernel $W_t$ is even and decays sufficiently fast at infinity for \eqref{102} and \eqref{102a} to hold. In particular, in this step, we could replace the Gaussian function $W$ with the Poisson kernel $P$ or with an even mollifier $\varphi\in C^\infty_c(\mathbb{R}^n)$. 
In contrast, Step 2 uses the properties of the Gaussian kernel $W_t$ in a crucial way.
\end{remark}

We next prove a preliminary version of Theorem \ref{theorem higher order} in the inhomogeneous case.

\begin{theorem}
\label{theorem higher order inhomogeneous}Let $m\in\mathbb{N}_{0}$ and $-1<a<m$. Suppose
that $f\in B^{m-a,1}(\mathbb{R}^{n})$ and let $u$ be given by
\eqref{function u}. Then
\begin{equation}
\int_{\mathbb{R}_{+}^{n+1}}t^{a}|\nabla^{m+1}u(x,t)|\;dtdx\lesssim\left\vert
f\right\vert _{B^{m-a,1}(\mathbb{R}^{n})}.\label{main estimate1}%
\end{equation}
\end{theorem}

\begin{proof}
[Proof of Theorem \ref{theorem higher order inhomogeneous}] Let $s:=m-a>0$ and $f\in
B^{s,1}(\mathbb{R}^{n})$.

\textbf{Step 1: } Assume that $s=k\in\mathbb{N}$. Then for every $\alpha
\in\mathbb{N}_{0}^{n}$ with $|\alpha|=k-1$, we have that $\frac{\partial
^{\alpha}f}{\partial x^{\alpha}}\in B^{1,1}(\mathbb{R}^{n})$. The goal is
then to show an $L^{1}$ bound for the entries of the tensor
\[
t^{a}\nabla^{m+1}u,
\]
where we recall that $u=W_{t}\ast f$. To reduce to the case where all the
derivatives are in the trace variable $x$, consider a multi-index
$(\beta,l)\in\mathbb{N}_{0}^{n}\times\mathbb{N}_{0}$, with $|\beta|+l=m+1$.
Since $|\beta|+2l=m+1+l>k-1$, by applying $l$ times formula (\ref{heat equation}) , we can write%
\[
t^{a}\left(  \frac{\partial^{\beta}}{\partial x^{\beta}}\left(  \frac
{\partial^{l}W_{t}}{\partial t^{l}}\right)  \right)  \ast f
\]
as linear combinations
of
\[
t^{a+l}\frac{\partial^{\gamma}W_{t}}{\partial x^{\gamma}}\ast\frac
{\partial^{\alpha}f}{\partial x^{\alpha}}%
\]
for $l=0,\ldots,m+1$, and where the multi-indices $\alpha,\gamma\in
\mathbb{N}_{0}^{n}$ satisfy $|\alpha|=k-1$ and $|\gamma|=m+l-k+2$.
As in \eqref{semi-group}, the
semi-group property of the heat extension leads to the identity
\[
\frac{\partial^{\gamma}W_{t}}{\partial x^{\gamma}}=\frac
{\partial^{\gamma^{\prime}}W_{t/\sqrt{2}}}{\partial x^{\gamma^{\prime}}}\ast
\frac{\partial^{2}W_{t/\sqrt{2}}}{\partial x_{i}\partial x_{j}}%
\]
for some $i,j=1,\ldots,n$ and $\gamma^{\prime}\in\mathbb{N}_{0}^{n}$ such that
$|\gamma^{\prime}|=|\gamma|-2=m+l-k$.\ This shows that one actually estimates
\[
t^{a+l}\frac{\partial^{\gamma^{\prime}}W_{t/\sqrt{2}}}{\partial x^{\gamma^{\prime}}%
}\ast\frac{\partial^{2}W_{t/\sqrt{2}}}{\partial x_{i}\partial x_{j}}\ast
\frac{\partial^{\alpha}f}{\partial x^{\alpha}}%
\]
for some $l\in\{0,\ldots,m+1\}$, and where the multi-indices $\alpha,\gamma
\in\mathbb{N}_{0}^{n}$ satisfy $|\alpha|=k-1$ and $|\gamma^{\prime}|=m+l-k$.

Since 
$\frac{\partial^{2}W_{t/\sqrt{2}}}{\partial x_{i}\partial x_{j}}$ is even and has
zero average in the trace variable $x$, as in \eqref{u1} we can write%
\begin{align*}
t^{a+l}&\left(  \frac{\partial^{\gamma^{\prime}}W_{t/\sqrt{2}}}{\partial
x^{\gamma^{\prime}}}\ast\frac{\partial^{2}W_{t/\sqrt{2}}}{\partial x_{i}\partial
x_{j}}\ast\frac{\partial^{\alpha}f}{\partial x^{\alpha}}\right)
(x) \\
&=\frac{t^{a+l}}{2}\int_{\mathbb{R}^{n}}\frac{\partial^{\gamma^{\prime}%
}W_{t/\sqrt{2}}}{\partial x^{\gamma^{\prime}}}(x-y)\int_{\mathbb{R}^{n}}%
\frac{\partial^{2}W_{t/\sqrt{2}}}{\partial x_{i}\partial x_{j}}(h)\Delta_{h}^{2}%
\frac{\partial^{\alpha}f}{\partial x^{\alpha}}(y-h)\,dhdy.
\end{align*}
The point is that when one integrates this quantity over $\mathbb{R}_{+}%
^{n+1}$, by\ Lemma \ref{lemma heat integral},
\[
\int_{\mathbb{R}^{n}}\left\vert \frac{\partial^{\gamma^{\prime}}W_{t/\sqrt{2}}%
}{\partial x^{\gamma^{\prime}}}(x)\right\vert dx\lesssim\frac{1}{t^{|\gamma
^{\prime}|}}%
\]
and therefore, by Tonelli's theorem, one has as an upper bound some universal
constant times
\[
\int_{\mathbb{R}^{n}}\int_{\mathbb{R}^{n}}\int_{0}^{\infty}t^{a+l-|\gamma
^{\prime}|}\left\vert \frac{\partial^{2}W_{t/\sqrt{2}}}{\partial x_{i}\partial x_{j}%
}(h)\right\vert \left\vert \Delta_{h}^{2}\frac{\partial^{\alpha}f}{\partial
x^{\alpha}}(w-h)\right\vert \,dtdhdw.
\]
By Lemma \ref{lemma heat integral} again,
\[
\int_{0}^{\infty}t^{a+l-|\gamma^{\prime}|}\left\vert \frac{\partial^{2}%
W_{t/\sqrt{2}}}{\partial x_{i}\partial x_{j}}(h)\right\vert \;dt\lesssim\frac
{1}{|h|^{n+2-a-l+|\gamma^{\prime}|-1}}=\frac{1}{|h|^{n+1}}.
\]
\textbf{Step 2: }Assume that $s\notin\mathbb{N}$ and let $k:=\lfloor s\rfloor
$. Then for every $\alpha\in\mathbb{N}_{0}^{n}$ with $|\alpha|=k$, we have
that $\frac{\partial^{\alpha}f}{\partial x^{\alpha}}\in B^{s-k,1}%
(\mathbb{R}^{n})$. As  in Step 1, to estimate the $L^{1}$ norm of $t^{a}%
\nabla^{m+1}u$, it suffices to estimate the $L^{1}$ norm of
\[
t^{a+l}\frac{\partial^{\gamma}W_{t}}{\partial x^{\gamma}}\ast\frac
{\partial^{\alpha}f}{\partial x^{\alpha}}%
\]
for $l=0,\ldots,m+1$, and where the multi-indices $\alpha,\gamma\in
\mathbb{N}_{0}^{n}$ satisfy $|\alpha|=k$ and $|\gamma|=m+1+l-k$. Since
$\int_{\mathbb{R}^{n}}\frac{\partial^{\gamma}W_{t}}{\partial x^{\gamma}%
}(h)\,dh=0$, we can write%
\begin{align}\label{1000}
\left(  \frac{\partial^{\gamma}W_{t}}{\partial x^{\gamma}}\ast\frac
{\partial^{\alpha}f}{\partial x^{\alpha}}\right)  (x)=\int_{\mathbb{R}^{n}%
}\frac{\partial^{\gamma}W_{t}}{\partial x^{\gamma}}(h)\Delta_{-h}%
\frac{\partial^{\alpha}f}{\partial x^{\alpha}}(x)\,dh.
\end{align}
Multiplying both sides by $t^{a+l}$ and integrating in $(x,t)$ gives%
\[
\int_{\mathbb{R}_{+}^{n+1}}t^{a+l}\left\vert \left(  \frac
{\partial^{\gamma}W_{t}}{\partial x^{\gamma}}\ast\frac{\partial^{\alpha}%
f}{\partial x^{\alpha}}\right)  (x)\right\vert \,dxdt\leq\int_{\mathbb{R}^{n}%
}\int_{\mathbb{R}^{n}}\int_{0}^{\infty}t^{a+l}\left\vert \frac{\partial
^{\gamma}W_{t}}{\partial x^{\gamma}}(h)\right\vert \left\vert \Delta_{-h}%
\frac{\partial^{\alpha}f}{\partial x^{\alpha}}(x)\right\vert \,dtdhdx,
\]
where we used Tonelli's theorem. By Lemma \ref{lemma heat integral},
\[
\int_{0}^{\infty}t^{a+l}\left\vert \frac{\partial^{\gamma}W_{t}}{\partial
x^{\gamma}}(h)\right\vert \,dt\lesssim\frac{1}{|h|^{n+|\gamma|-a-l-1}}=\frac
{1}{|h|^{n+s-k}}.
\]
In turn,%
\[
\int_{\mathbb{R}_{+}^{n+1}}t^{a+l}\left\vert \left(  \frac
{\partial^{\gamma}W_{t}}{\partial x^{\gamma}}\ast\frac{\partial^{\alpha}%
f}{\partial x^{\alpha}}\right)  (x)\right\vert dxdt\lesssim\int_{\mathbb{R}%
^{n}}\int_{\mathbb{R}^{n}}\left\vert \Delta_{-h}\frac{\partial^{\alpha}%
f}{\partial x^{\alpha}}(x)\right\vert \frac{dh}{|h|^{n+s-k}}dx.
\]

\end{proof}

We next prove a preliminary version of Theorem \ref{theorem p} in the inhomogeneous case.

\begin{theorem}
\label{theorem p inhomogeneous}Let $m\in\mathbb{N}_{0}$, $1\leq p<\infty$, and
$-1<a<p(m+1)-1$. Suppose that $f\in B^{m+1-(a+1)/p,p}(\mathbb{R}^{n})$
and let $u$ be given by \eqref{function u}. Then
\[
\int_{\mathbb{R}_{+}^{n+1}}t^{a}|\nabla^{m+1}u(x,t)|^{p}dxdt\lesssim\left\vert
f\right\vert _{B^{m+1-(a+1)/p,p}(\mathbb{R}^{n})}^{p}.
\]

\end{theorem}

\begin{proof}
[Proof of Theorem \ref{theorem p inhomogeneous}]Let $s:=m+1-(a+1)/p>0$ and $f\in B^{s,p}(\mathbb{R}^{n})$. 
\textbf{Step 1: }
Assume that $s=k\in\mathbb{N}$. Then for every $\alpha
\in\mathbb{N}_{0}^{n}$ with $|\alpha|=k-1$, we have that $\frac{\partial
^{\alpha}f}{\partial x^{\alpha}}\in B^{1,p}(\mathbb{R}^{n})$. As in Step
1 of the proof of Theorem \ref{theorem higher order}, to obtain an $L^{p}$
bound for the entries of the tensor $t^{a/p}\nabla^{m+1}u$, it suffices to
estimate the $L^{p}$ norm of
\[
t^{a/p+l}\frac{\partial^{\gamma^{\prime}}W_{t/\sqrt{2}}}{\partial x^{\gamma^{\prime}%
}}\ast\frac{\partial^{2}W_{t/\sqrt{2}}}{\partial x_{i}\partial x_{j}}\ast
\frac{\partial^{\alpha}f}{\partial x^{\alpha}}%
\]
for some $l\in\{0,\ldots,m+1\}$, and where the multi-indices $\alpha,\gamma
\in\mathbb{N}_{0}^{n}$ satisfy $|\alpha|=k-1$ and $|\gamma^{\prime}|=m+l-k$.

Let
\begin{align*}
g(x):=\frac{1}{2}\int_{\mathbb{R}^{n}}\frac{\partial^{2}W_{t/\sqrt{2}}}{\partial
x_{j}^{2}}(h)\Delta_{h}^{2}f(x-h)\,dh.
\end{align*} 
By Young's
inequality for convolutions,%
\begin{align}
t^{a/p+l} &  \left\Vert \frac{\partial^{\gamma^{\prime}}W_{t/\sqrt{2}}}{\partial
x^{\gamma^{\prime}}}\ast\frac{\partial^{2}W_{t/\sqrt{2}}}{\partial x_{i}\partial
x_{j}}\ast\frac{\partial^{\alpha}f}{\partial x^{\alpha}}\right\Vert
_{L^{p}(\mathbb{R}^{n})}=t^{a/p+l}\left\Vert \frac{\partial^{\gamma^{\prime}%
}W_{t/\sqrt{2}}}{\partial x^{\gamma^{\prime}}}\ast g\right\Vert _{L^{p}%
(\mathbb{R}^{n})}\label{c2}\\
&  \leq t^{a/p+l}\left\Vert \frac{\partial^{\gamma^{\prime}}W_{t/\sqrt{2}}}{\partial
x^{\gamma^{\prime}}}\right\Vert _{L^{1}(\mathbb{R}^{n})}\left\Vert
g\right\Vert _{L^{p}(\mathbb{R}^{n})}\lesssim t^{a/p+l-|\gamma^{\prime}%
|}\left\Vert g\right\Vert _{L^{p}(\mathbb{R}^{n})},\nonumber
\end{align}
where we used Lemma \ref{lemma heat integral}. To estimate $\left\Vert
g\right\Vert _{L^{p}(\mathbb{R}^{n})}$, we write $\left\vert \frac
{\partial^{2}W_{t/\sqrt{2}}}{\partial x_{i}\partial x_{j}}\right\vert =\left\vert
\frac{\partial^{2}W_{t/\sqrt{2}}}{\partial x_{i}\partial x_{j}}\right\vert
^{1/p^{\prime}}\left\vert \frac{\partial^{2}W_{t/\sqrt{2}}}{\partial x_{i}\partial
x_{j}}\right\vert ^{1/p}$. By H\"{o}lder's inequality%
\begin{align*}
&  \left\vert g(x)\right\vert \leq\left(  \int_{\mathbb{R}^{n}}\left\vert
\frac{\partial^{2}W_{t/\sqrt{2}}}{\partial x_{i}\partial x_{j}}(h)\right\vert
dh\right)  ^{1/p^{\prime}}\left(  \int_{\mathbb{R}^{n}}\left\vert
\frac{\partial^{2}W_{t/\sqrt{2}}}{\partial x_{i}\partial x_{j}}(h)\right\vert
\left\vert \Delta_{h}^{2}\frac{\partial^{\alpha}f}{\partial x^{\alpha}%
}(x-h)\right\vert ^{p}dh\right)  ^{1/p}\\
&  \lesssim t^{-2/p^{\prime}}\left(  \int_{\mathbb{R}^{n}}\left\vert
\frac{\partial^{2}W_{t/\sqrt{2}}}{\partial x_{i}\partial x_{j}}(h)\right\vert
\left\vert \Delta_{h}^{2}\frac{\partial^{\alpha}f}{\partial x^{\alpha}%
}(x-h)\right\vert ^{p}dh\right)  ^{1/p},
\end{align*}
where we used again Lemma \ref{lemma heat integral}. Hence,%
\begin{equation}
\int_{\mathbb{R}^{n}}|g(x)|^{p}dx\lesssim t^{-2(p-1)}\int_{\mathbb{R}^{n}}%
\int_{\mathbb{R}^{n}}\left\vert \frac{\partial^{2}W_{t/\sqrt{2}}}{\partial
x_{i}\partial x_{j}}(h)\right\vert \left\vert \Delta_{h}^{2}\frac
{\partial^{\alpha}f}{\partial x^{\alpha}}(x-h)\right\vert ^{p}dhdx.\label{c3}%
\end{equation}
Raising both sides of (\ref{c2}) to power $p$, integrating in $t$, and
using (\ref{c3}) gives%
\begin{align*}
&  \int_{\mathbb{R}_{+}^{n+1}}t^{a+lp}\left\vert \left(
\frac{\partial^{\gamma^{\prime}}W_{t/\sqrt{2}}}{\partial x^{\gamma^{\prime}}}%
\ast\frac{\partial^{2}W_{t/\sqrt{2}}}{\partial x_{i}\partial x_{j}}\ast\frac
{\partial^{\alpha}f}{\partial x^{\alpha}}\right)  (x)\right\vert ^{p}dxdt\\
&  \lesssim\int_{\mathbb{R}^{n}}\int_{\mathbb{R}^{n}}\int_{0}^{\infty
}t^{a+lp-|\gamma^{\prime}|p-2(p-1)}\left\vert \frac{\partial^{2}W_{t/\sqrt{2}}%
}{\partial x_{i}\partial x_{j}}(h)\right\vert \left\vert \Delta_{h}^{2}%
\frac{\partial^{\alpha}f}{\partial x^{\alpha}}(x-h)\right\vert ^{p}dtdhdx.
\end{align*}
By Lemma \ref{lemma heat integral},
\[
\int_{0}^{\infty}t^{a+lp-|\gamma^{\prime}|p-2(p-1)}\left\vert \frac
{\partial^{2}W_{t/\sqrt{2}}}{\partial x_{i}\partial x_{j}}(h)\right\vert
\,dt\lesssim\frac{1}{|h|^{n+2-a-lp+|\gamma^{\prime}|p+2(p-1)-1}}=\frac
{1}{|h|^{n+p}}.
\]
In turn,%
\begin{align*}
\int_{\mathbb{R}_{+}^{n+1}}t^{a+lp}&\left\vert \left(
\frac{\partial^{\gamma^{\prime}}W_{t/\sqrt{2}}}{\partial x^{\gamma^{\prime}}}%
\ast\frac{\partial^{2}W_{t/\sqrt{2}}}{\partial x_{i}\partial x_{j}}\ast\frac
{\partial^{\alpha}f}{\partial x^{\alpha}}\right)  (x)\right\vert
^{p}dxdt\\
&\lesssim\int_{\mathbb{R}^{n}}\int_{\mathbb{R}^{n}}\left\vert \Delta
_{h}^{2}\frac{\partial^{\alpha}f}{\partial x^{\alpha}}(x-h)\right\vert ^{p}%
\frac{dh}{|h|^{n+p}}dx.
\end{align*}
\textbf{Step 2: } Assume that $s\notin\mathbb{N}$ and let $k:=\lfloor s\rfloor
$. Then for every $\alpha\in\mathbb{N}_{0}^{n}$ with $|\alpha|=k$, we have
that $\frac{\partial^{\alpha}f}{\partial x^{\alpha}}\in B^{s-k,p}%
(\mathbb{R}^{n})$. As in Step 2 of the proof of Theorem
\ref{theorem higher order}, to obtain $L^{p}$ bound for the entries of the
tensor $t^{a/p}\nabla^{m+1}u$, it suffices to estimate
\[
t^{a/p+l}\frac{\partial^{\gamma}W_{t}}{\partial x^{\gamma}}\ast\frac
{\partial^{\alpha}f}{\partial x^{\alpha}}%
\]
for $l=0,\ldots,m+1$, and where the multi-indices $\alpha,\gamma\in
\mathbb{N}_{0}^{n}$ satisfy $|\alpha|=k$ and $|\gamma|=m+1+l-k$. Writing
$\left\vert \frac{\partial^{\gamma}W_{t}}{\partial x^{\gamma}}\right\vert
=\left\vert \frac{\partial^{\gamma}W_{t}}{\partial x^{\gamma}}\right\vert
^{1/p^{\prime}}\left\vert \frac{\partial^{\gamma}W_{t}}{\partial x^{\gamma}%
}\right\vert ^{1/p}$, by H\"{o}lder's inequality and \eqref{1000}%
\begin{align*}
&  \left\vert \left(  \frac{\partial^{\gamma}W_{t}}{\partial x^{\gamma}}%
\ast\frac{\partial^{\alpha}f}{\partial x^{\alpha}}\right)  (x)\right\vert \\
&  \leq\left(  \int_{\mathbb{R}^{n}}\left\vert \frac{\partial^{\gamma}W_{t}%
}{\partial x^{\gamma}}(h)\right\vert dh\right)  ^{1/p^{\prime}}\left(
\int_{\mathbb{R}^{n}}\left\vert \frac{\partial^{\gamma}W_{t}}{\partial
x^{\gamma}}(h)\right\vert \left\vert \Delta_{-h}\frac{\partial^{\alpha}%
f}{\partial x^{\alpha}}(x)\right\vert ^{p}dh\right)  ^{1/p}\\
&  \lesssim t^{-|\gamma|/p^{\prime}}\left(  \int_{\mathbb{R}^{n}}\left\vert
\frac{\partial^{\gamma}W_{t}}{\partial x^{\gamma}}(h)\right\vert \left\vert
\Delta_{-h}\frac{\partial^{\alpha}f}{\partial x^{\alpha}}(x)\right\vert
^{p}dh\right)  ^{1/p},
\end{align*}
where we used Lemma \ref{lemma heat integral}. Multiplying both sides by
$t^{a/p+l}$, raising to power $p$, and integrating in $(x,t)$ gives%
\begin{align*}
&  \int_{\mathbb{R}_{+}^{n+1}}t^{a+lp}\left\vert \left(
\frac{\partial^{\gamma}W_{t}}{\partial x^{\gamma}}\ast\frac{\partial^{\alpha
}f}{\partial x^{\alpha}}\right)  (x)\right\vert ^{p}dxdt\\
&  \lesssim\int_{\mathbb{R}^{n}}\int_{\mathbb{R}^{n}}\int_{0}^{\infty
}t^{a+lp-(p-1)|\gamma|}\left\vert \frac{\partial^{\gamma}W_{t}}{\partial
x^{\gamma}}(h)\right\vert \left\vert \Delta_{-h}\frac{\partial^{\alpha}%
f}{\partial x^{\alpha}}(x)\right\vert ^{p}dtdhdx.
\end{align*}
By Lemma \ref{lemma heat integral},
\[
\int_{0}^{\infty}t^{a+lp-(p-1)|\gamma|}\left\vert \frac{\partial^{\gamma}%
W_{t}}{\partial x^{\gamma}}(h)\right\vert \,dt\lesssim\frac{1}{|h|^{n+|\gamma
|-a-lp+(p-1)|\gamma|-1}}=\frac{1}{|h|^{n+(s-k)p}}.
\]
In turn,
\[
\int_{\mathbb{R}_{+}^{n+1}}t^{a+lp}\left\vert \left(
\frac{\partial^{\gamma}W_{t}}{\partial x^{\gamma}}\ast\frac{\partial^{\alpha
}f}{\partial x^{\alpha}}\right)  (x)\right\vert ^{p}dxdt\lesssim\int%
_{\mathbb{R}^{n}}\int_{\mathbb{R}^{n}}\left\vert \Delta_{-h}\frac
{\partial^{\alpha}f}{\partial x^{\alpha}}(x)\right\vert ^{p}\frac
{dh}{|h|^{n+(s-k)p}}dx.
\]

\end{proof}

%

We conclude this section with the proofs of Theorems \ref{theorem higher order} and \ref{theorem p}.

\begin{proof}[Proof of Theorems \ref{theorem higher order} and \ref{theorem p}]

The only place that the assumption $f\in B^{m+1-(a+1)/p,p}(\mathbb{R}^{n})$ was used in the proofs of Theorems  \ref{theorem higher order} and  \ref{theorem p} was to ensure that $W_t*f$ is well-defined.  The proofs of Theorems \ref{theorem higher order} and \ref{theorem p} will therefore be complete if we can show this convolution is well-defined for $f\in \dot{B}^{m+1-(a+1)/p,p}(\mathbb{R}^{n})$. 

By \cite[Remark 17.27 on p.~556]{leoni-book-sobolev}, for $f\in \dot{B}^{m+1-(a+1)/p,p}(\mathbb{R}^{n})$ one has
\begin{align*}
f=u+v
\end{align*}
for $u \in L^p(\mathbb{R}^{n})$ and $v \in \dot{W}^{l,p}(\mathbb{R}^{n})$
 where we can take $l\in \mathbb{N}$ so large that $lp>n$.  As $W_t*u$ is well-defined for $u \in L^p(\mathbb{R}^{n})$, it remains to show that $W_t*v$ is well-defined. This follows from the fact that $v$ has polynomial growth by Theorem \ref{theorem growth} in the appendix below.

\end{proof}

\section{The Inhomogeneous Case}
\label{section non-homogeneous}
We now turn to the proof of Theorem \ref{theorem non-homogeneous}. The proof of Step 2 is an adaptation of Mironescu's
argument who studied the case $m=0$ \cite{mironescu2015}.

\begin{proof}
[Proof of Theorem \ref{theorem non-homogeneous}] \textbf{Step 1: } Assume that $a<m$. Let $\psi\in C^{\infty
}([0,\infty))$ be a nonnegative decreasing function such that $\psi=1$ in
$[0,1]$, $\psi(t)=0$ for $t\geq2$ and define $F(x,t)=\psi(t)u(x,t)$, where
$u=W_{t}\ast f$. By Tonelli's theorem and the change of variables
$z=yt^{-1}$,%
\[
\int_{\mathbb{R}^{n}}|u(x,t)|\,dx\leq\int_{\mathbb{R}^{n}}|f(x)|\,dx\int%
_{\mathbb{R}^{n}}W(z)\,dz=\int_{\mathbb{R}^{n}}|f(x)|\,dx.
\]
Since $0\leq\psi\leq1$ and $\psi(t)=0$ for $t\geq2$, it follows%
\begin{equation}
\int_{\mathbb{R}_{+}^{n+1}}t^{a}|F(x,t)|\,dxdt=\int_{0}^{\infty
}t^{a}\psi(t)\int_{\mathbb{R}^{n}}|u(x,t)|\,dxdt\leq\int_{0}^{2}t^{a}%
dt\int_{\mathbb{R}^{n}}|f(x)|\,dx.\label{be 400}%
\end{equation}
Observe that the first integral on the right-hand side is finite because
$a>-1$.

Since $\psi=1$ in $[0,1]$, $F(x,t)=u(x,t)$ for $t\leq1$, and so
$\operatorname*{Tr}(F)=\operatorname*{Tr}(u)=f$. It remains to estimate the
derivatives of $F$. Consider a multi-index $(\beta,l)\in\mathbb{N}_{0}%
^{n}\times\mathbb{N}_{0}$, with $|\beta|+l=m+1$. By the product rule%
\[
\frac{\partial^{\beta}}{\partial x^{\beta}}\left(  \frac{\partial^{l}%
F}{\partial t^{l}}\right)
\]
can be written as a linear combination of $\psi^{(l-j)}(t)\frac{\partial
^{\beta}}{\partial x^{\beta}}\left(  \frac{\partial^{j}u}{\partial t^{j}%
}\right)  $ for $j=0,\ldots,l$. As in Step 1 of the proof of Theorem
\ref{theorem higher order} we can use (\ref{heat equation}) to write
$\frac{\partial^{\beta}}{\partial x^{\beta}}\left(  \frac{\partial^{j}%
u}{\partial t^{j}}\right)  $ as a linear combination of $t^{i}\frac
{\partial^{\gamma}u}{\partial x^{\gamma}}$ with $i \in \{0,\ldots, j\}$, where the multi-index $\gamma
\in\mathbb{N}_{0}^{n}$ satisfies $|\gamma|=|\beta|+2i=m+1-l+2i$. There are now
two cases. If $-l+2i<0$, we use the Gagliardo--Nirenberg interpolation
inequality \cite[Theorem 12.85]{leoni-book-sobolev} to estimate%
\[
\int_{\mathbb{R}^{n}}|\nabla_x^{m+1-l+2i}u(x,t)|\,dx\lesssim\int_{\mathbb{R}^{n}%
}|u(x,t)|\,dx+\int_{\mathbb{R}^{n}}|\nabla_x^{m+1}u(x,t)|\,dx.
\]
In turn,%
\begin{align*}
\int_{\mathbb{R}_{+}^{n+1}}&t^{a+i}|\psi^{(l-j)}(t)||\nabla
_x^{m+1-l+2i}u(x,t)|\,dxdt \\ & \lesssim\int_{0}^{2}t^{a+i}\int_{\mathbb{R}^{n}%
}|u(x,t)|\,dxdt+\int_{0}^{2}t^{a+i}\int_{\mathbb{R}^{n}}|\nabla_x^{m+1}%
u(x,t)|\,dxdt\\
& \lesssim\int_{0}^{2}t^{a}\int_{\mathbb{R}^{n}}|u(x,t)|\,dxdt+\int_{0}%
^{2}t^{a}\int_{\mathbb{R}^{n}}|\nabla_x^{m+1}u(x,t)|\,dxdt\\
& \lesssim\int_{\mathbb{R}^{n}}|f(x)|\,dx+\left\vert f\right\vert
_{B^{m-a,1}(\mathbb{R}^{n})}%
\end{align*}
by (\ref{main estimate}) and (\ref{be 400}). 

If $-l+2i\geq0$, we use (\ref{main estimate}) with $a$ replaced by $a+i$ and
$m$  by $m-l+2i$ to obtain%
\[
\int_{\mathbb{R}_{+}^{n+1}}t^{a+i}|\nabla^{m+1-l+2i}u(x,t)|\;dtdx\lesssim
\left\vert f\right\vert _{B^{m-a-l+i,1}(\mathbb{R}^{n})}.
\]
If $i=l$, we are done since $f\in B^{m-a,1}(\mathbb{R}^{n})$. Otherwise, we
use the fact that by Proposition \ref{proposition embedding},
\[
\left\vert f\right\vert _{B^{m-a-l+i,1}(\mathbb{R}^{n})}\lesssim\Vert
f\Vert_{L^{1}(\mathbb{R}^{n})}+\Vert\nabla_x^{\lfloor m-a-l+i\rfloor+1}f\Vert_{L^{1}%
(\mathbb{R}^{n})}\le \Vert f\Vert_{B^{m-a,1}(\mathbb{R}^{n})}.
\]
Similar estimates hold when $1\le |\beta|+l<m+1$. We omit the details. 

\textbf{Step 2: }Assume that $a=m$ and let
$f\in C_{c}^{\infty}(\mathbb{R}^{n})$. Consider a function $\varphi\in
C^{\infty}([0,\infty))$ such that $0\leq\varphi\leq1$, $\varphi=1$ in
$[0,1]$, and $\varphi=0$ in $[2,\infty)$. For $l\in\mathbb{N}$ define
$v_{l}(x,t):=f(x)\varphi(lt)$. By the properties of $\varphi$, $v_{l}%
(x,0)=f(x)$. The product rule implies that the entries of the tensor
$\nabla^{m+1}v_{l}$ can be written as linear combinations of the functions
\[
l^{i}\varphi^{(i)}(lt)\frac{\partial^{\beta}f}{\partial x^{\beta}}(x)
\]
for multi-indices $\beta\in\mathbb{N}_{0}^{n}$ and $i=0,\ldots,m+1-|\beta|$ if
$\beta\neq0$ or $i=1,\ldots,m+1$ if $\beta=0$. This, along with the change of
variables $r=lt$ leads to the estimate
\begin{align}
\int_{\mathbb{R}_{+}^{n+1}} &  t^{m}|\nabla^{m+1}v_{l}(x,t)|\,dxdt\nonumber\\
&  \lesssim\sum_{|\beta|=1}^{m+1}\sum_{i=0}^{m+1-|\beta|}\frac{1}{l^{m-i+1}%
}\int_{\mathbb{R}^{n}}\left\vert \frac{\partial^{\beta}f}{\partial x^{\beta}%
}(x)\right\vert \,dx\int_{0}^{\infty}r^{m}|\varphi^{(i)}(r)|\,dr\label{m100}\\
&  \;\;+\int_{\mathbb{R}^{n}}|f(x)|\,dx\int_{0}^{\infty}r^{m}|\varphi
^{(m+1)}(r)|\,dr.\label{m101}%
\end{align}
One observes that the quantity numbered by equation \eqref{m100} tends to zero as $l\rightarrow\infty$. Thus, if $f\neq0$, by taking $l$ large enough, we
can majorize the quantity \eqref{m100} by $\Vert f\Vert_{L^{1}(\mathbb{R}%
^{n})}$, while the quantity numbered by equation \eqref{m101} is just a
constant multiple of this norm. If $f=0$, we take $u=0$. This proves
\eqref{z1} in the case $f\in C_{c}^{\infty}(\mathbb{R}^{n})$. The general case
follows from the density of $C_{c}^{\infty}(\mathbb{R}^{n})$ in $L^{1}%
(\mathbb{R}^{n})$. We omit the details.
\end{proof}

\begin{remark} By taking $a=0$ in \eqref{be 400}, we have that $F\in \dot{W}_{}^{m+1,1}(\mathbb{R}_{+}^{n+1})\cap
L^{1}(\mathbb{R}^{n+1}_+)$, with 
$$|F|_{W_{}^{m+1,1}(\mathbb{R}_{+}^{n+1})}+\Vert F\Vert_{L^1(\mathbb{R}^{n+1}_+)}\lesssim\Vert f\Vert_{B^{m-a,1}(\mathbb{R}^{n})}.$$    
\end{remark}
\begin{remark}
Using the fact that for $v\in W_{m}^{m+1,1}(\mathbb{R}_{+})\cap C^{\infty
}([0,\infty))$, we have%
\[
v(0)=c\int_{0}^{\infty}t^{m}\frac{d^{(m+1)}v}{dt^{m+1}}(t)\,dt,
\]
for all $u\in W_{m}^{m+1,1}(\mathbb{R}_{+}^{n+1})\cap C^{\infty}(\mathbb{R}^{n}\times[0,\infty))$, we can write%
\[
u(x,0)=c\int_{0}^{\infty}t^{m}\frac{\partial^{(m+1)}u}{\partial t^{m+1}%
}(x,t)\,dt,
\]
and so,%
\[
\int_{\mathbb{R}^{n}}|u(x,0)|\,dx\lesssim\int_{0}^{\infty}\int_{\mathbb{R}^{n}%
}t^{m}\left\vert \frac{\partial^{(m+1)}u}{\partial t^{m+1}}(x,t)\right\vert
\,dxdt.
\]
By a reflection (see, e.g., \cite[Exercise 13.3]{leoni-book-sobolev}) and mollification argument, we have that for every function $u\in W_{m}^{m+1,1}%
(\mathbb{R}_{+}^{n+1})$, the trace of $u$ belongs to $L^{1}(\mathbb{R}^{n})$.
Together with the previous theorem, this shows that%
\[
\operatorname*{Tr}(W_{m}^{m+1,1}(\mathbb{R}_{+}^{n+1}))=L^{1}(\mathbb{R}^{n}).
\]

\end{remark}

When $k=0$, the proof of the following lemma is due to
Gmeineder, Raita, and Van Schaftingen
\cite{gmeineder-raita-van-schaftingen2022} and is an adaptation of Mironescu's
argument in Step 2 of the proof of Theorem \ref{theorem non-homogeneous} above. See also the paper of Demengel \cite{demengel-1984} and \cite[Theorem 18.43]{leoni-book-sobolev} for an alternative proof based on Gagliardo's original proof \cite{gagliardo1957}.
\begin{lemma}
\label{lemma normal L1}Let $m\in\mathbb{N}$ and $k\in\mathbb{N}_{0}$ with
$k<m$. Suppose that $g\in L^{1}(\mathbb{R}^{n})$. Then there exists $G\in
\dot{W}_{k}^{m+1,1}(\mathbb{R}_{+}^{n+1})$ such that $\operatorname*{Tr}%
(\,G)=0$, $\operatorname*{Tr}(\,\frac{\partial^{j}G}{\partial t^{j}})=0$ for
$j=1,\ldots,m-k-1$, $\operatorname*{Tr}(\,\frac{\partial^{m-k}G}{\partial
t^{m-k}})=g$, and%
\begin{equation}
\int_{\mathbb{R}_{+}^{n+1}}t^{k}|\nabla^{m+1}G(x,t)|\,dxdt\lesssim\Vert
g\Vert_{L^{1}(\mathbb{R}^{n})}.\label{estimate L1}%
\end{equation}

\end{lemma}

\begin{proof}
Assume first that $g\in C_{c}^{\infty}(\mathbb{R}^{n})$ and let $\varphi\in
C_{c}^{\infty}([0,\infty))$ be such that $\varphi(0)=1$, $\varphi^{\prime
}(0)=\cdots=\varphi^{(m-k)}(0)=1$. For $n\in\mathbb{N}$ define $v_{l}%
(x,t):=g(x)\frac{t^{m-k}}{(m-k)!}\varphi(lt)$. By the properties of $\varphi$,
$v_{l}(x,0)=0$, $\frac{\partial^{j}v_{l}}{\partial t^{j}}(x,0)=0$ for
$j=1,\ldots,m-k-1$, and $\frac{\partial^{m-k}v_{l}}{\partial t^{m-k}%
}(x,0)=g(x)$. The product rule implies that the entries of the tensor
$\nabla^{m+1}v_{l}$ can be written as linear combinations of the functions
\[
l^{i} t^{(|\beta|-k-1+i)_{+}}\varphi^{(i)}(lt)\frac{\partial^{\beta}g}{\partial
x^{\beta}}(x)
\]
for multi-indices $\beta\in\mathbb{N}_{0}^{n}$ and $i=0,\ldots,m+1-|\beta|$ if
$\beta\neq0$ or $i=1,\ldots,m+1$ if $\beta=0$. Here, $s_{+}$ is the positive
part of $s$. This, along with the change of variables $r=lt$, leads to the
estimate
\begin{align}
&\int_{\mathbb{R}_{+}^{n+1}}   t^{k}|\nabla^{m+1}v_{l}(x,t)|\,dxdt\nonumber\\
&  \lesssim\sum_{|\beta|=1}^{m+1}\sum_{i=0}^{m+1-|\beta|}\frac{l^{i}%
}{l^{k+(|\beta|-k-1+i)_{+}+1}}\int_{\mathbb{R}^{n}}\left\vert \frac
{\partial^{\beta}g}{\partial x^{\beta}}(x)\right\vert \,dx\int_{0}^{\infty
}r^{k+(|\beta|-k-1+i)_{+}}|\varphi^{(i)}(r)|\,dr\label{m6}\\
&  \;\;+\int_{\mathbb{R}^{n}}|g(x)|\,dx\sum_{i=1}^{m+1}\frac{l^{i}%
}{l^{k+(i-1-k)_{+}+1}}\int_{0}^{\infty}r^{k+(i-1-k)_{+}}|\varphi
^{(i)}(r)|\,dr.\label{m6'}%
\end{align}
If $|\beta|-k-1+i\geq0$, then $\frac{l^{i}}{l^{k+(|\beta|-k-1+i)_{+}+1}}%
=\frac{1}{l^{|\beta|}}\rightarrow0$ since $|\beta|\geq1$, while if
$|\beta|-k-1+i<0$, then $\frac{l^{i}}{l^{k+(|\beta|-k-1+i)_{+}+1}}=\frac
{l^{i}}{l^{k+1}}\rightarrow0$. Hence, the quantity
numbered by equation \eqref{m6} tends to zero for as $l\rightarrow\infty$.
Thus, if $g\neq0$, by taking $l$ large enough, we can majorize the quantity
\eqref{m6} by $\Vert g\Vert_{L^{1}(\mathbb{R}^{n})}$. 

On the other hand, if $i-1-k\geq0$, $\frac{l^{i}}{l^{k+(i-1-k)_{+}+1}}=1$,
while if $i-1-k<0$, then $\frac{l^{i}}{l^{k+(i-1-k)_{+}+1}}\leq1$. Hence, the
quantity numbered by equation \eqref{m6'} is bounded from above by a constant
multiple of $\Vert g\Vert_{L^{1}(\mathbb{R}^{n})}$. If $g=0$, we take $G=0$.
This proves \eqref{estimate L1} in the case $g\in C_{c}^{\infty}%
(\mathbb{R}^{n})$. The general case follows from the density of $C_{c}%
^{\infty}(\mathbb{R}^{n})$ in $L^{1}(\mathbb{R}^{n})$. We omit the details.
\end{proof}

\begin{proof}
[Proof of Theorem \ref{theorem trace higher}] 
\textbf{Step 1: }If
$a\in\mathbb{N}_{0}$ let $l:=m-a-1$, while if $a\notin\mathbb{N}_{0}$ let
$l:=\lfloor m-a\rfloor$. Assume that $f_{j}\in B^{m-a-j,1}(\mathbb{R}^{n})$ for
$j=1,\ldots,l$. We first prove that there exists a function $u_{j}\in\dot
{W}_{a}^{m+1,1}(\mathbb{R}_{+}^{n+1})$ such that $\operatorname*{Tr}%
(\,u_{j})=0$, $\operatorname*{Tr}(\,\frac{\partial^{i}u_{j}}{\partial t^{i}%
})=0$ for all $i=1,\ldots,j-1$, and $\operatorname*{Tr}%
(\,\frac{\partial^{j}u_{j}}{\partial t^{j}})=f_{j}$, with%
\begin{equation}
\Vert\nabla^{m+1}u_{j}\Vert_{L_{a}^{1}(\mathbb{R}_{+}^{n+1}%
)}\lesssim
|f_{j}|_{B^{m-a-j,1}(\mathbb{R}^{n})}\label{m1}%
\end{equation}
(see \eqref{Lp weighted}).
Define%
\[
u_{j}(x,t):=\frac{t^{j}}{j!}(W_{t}\ast f_{j})(x),
\]
where $W$ is the Gaussian function (\ref{gaussian}). The desired properties
$\operatorname*{Tr}(\,\frac{\partial^{j}u_{j}}{\partial t^{j}})=f_{j}$,
$\operatorname*{Tr}(\,u_{j})=0$, $\operatorname*{Tr}(\,\frac{\partial^{i}%
u_{j}}{\partial t^{i}})=0$ for all $i=1,\ldots,j-1$ can be checked by the
properties of $W_{t}$. Concerning the estimate \eqref{m1}, one observes that
the product rule implies that the entries of the tensor $\nabla^{m+1}u_{j}$
are linear combinations of the entries of the tensor $t^{j-i}\nabla
^{m+1-i}W_{t}\ast f_{j}$. Thus the estimate \eqref{m1} is a consequence of
(\ref{main estimate}) applied to the function $u=W_{t}\ast f_{j}$, with $m$
replaced by $m-i$ and $a$ by $a+j-i$, which asserts that one has the
estimates
\[
\int_{\mathbb{R}_{+}^{n+1}}t^{a+j-i}|\nabla^{m+1-i}\left(W_{t}\ast f_{j}%
\right)|\;dtdx\lesssim\left\vert f_{j}\right\vert _{B^{m-a-j,1}(\mathbb{R}^{n})}%
\]
for $i=0,\ldots,j$.

 \textbf{Step 2: } We are now ready to prove the general case. Assume first that
$a=k\in\mathbb{N}_{0}$ let $l:=m-k-1$. We will use the fact that if $u\in
W_{k}^{m+1,1}(\mathbb{R}_{+}^{n+1})$, then $\operatorname*{Tr}(\,u)\in
B^{m-k,1}(\mathbb{R}^{n})$, $\operatorname*{Tr}(\,\frac{\partial^{j}%
u}{\partial t^{j}})\in B^{m-k-j,1}(\mathbb{R}^{n})$ for $j=1,\ldots,l$,
$\operatorname*{Tr}(\,\frac{\partial^{m-k}u}{\partial t^{m-k}})\in
L^{1}(\mathbb{R}^{n})$, with%
\begin{align}
|\operatorname*{Tr}(\,u)|_{B^{m-k,1}(\mathbb{R}^{n})} &  +\sum_{j=1}%
^{l}\left\vert \operatorname*{Tr}\left(  \frac{\partial^{j}u}{\partial t^{j}%
}\right)  \right\vert _{B^{m-k-j,1}(\mathbb{R}^{n})}+\left\Vert
\operatorname*{Tr}\left(  \frac{\partial^{m-k}u}{\partial t^{m-k}}\right)
\right\Vert _{L^{1}(\mathbb{R}_{+}^{n+1})}\label{m7}\\
&  \lesssim\Vert\nabla^{m+1}u_{m}\Vert_{L_{k}^{1}(\mathbb{R}_{+}^{n+1}%
)}\nonumber
\end{align}
(see \cite[Theorem 18.57]{leoni-book-sobolev}). By Theorem \ref{theorem non-homogeneous}, there exists $v_{0}\in
\dot{W}_{k}^{m+1,1}(\mathbb{R}_{+}^{n+1})$ such that $\operatorname*{Tr}%
(\,v_{0})=f_{0}$ and
\begin{equation}
\Vert\nabla^{m+1}v_{0}\Vert_{L_{k}^{1}(\mathbb{R}_{+}^{n+1})}\lesssim
|f_{0}|_{B^{m-k,1}(\mathbb{R}^{n})}.\label{m8}%
\end{equation}
In turn, (\ref{m7}) holds for $v_{0}$. Hence, we can apply Step 1, with
$f_{1}$ replaced by $f_{1}-\operatorname*{Tr}\left(  \frac{\partial v_{0}%
}{\partial t}\right)  $, to find a function $v_{1}\in\dot{W}^{m+1,1}%
(\mathbb{R}_{+}^{n+1})$ such that $\operatorname*{Tr}(\,v_{1})=0$, $\operatorname*{Tr}(\,\frac{\partial v_{1}%
}{\partial t})=f_{1}-\operatorname*{Tr}\left(  \frac{\partial v_{0}}{\partial
t}\right)  $, and
\begin{align*}
\Vert\nabla^{m+1}v_{1}\Vert_{L_{k}^{1}(\mathbb{R}_{+}^{n+1})}  & \lesssim
|f_{1}|_{B^{m-1,1}(\mathbb{R}^{n})}+\left\vert \operatorname*{Tr}\left(
\frac{\partial v_{0}}{\partial t}\right)  \right\vert _{B^{m-1,1}%
(\mathbb{R}^{n})}\\
& \lesssim|f_{1}|_{B^{m-1,1}(\mathbb{R}^{n})}+|f_{0}|_{B^{m,1}(\mathbb{R}^{n}%
)},
\end{align*}
where the last inequality follows from (\ref{m7}), with $v_{0}$ in place of $u$,
and (\ref{m8}).

Inductively, assume that $v_{j}\in\dot{W}_{k}^{m+1,1}(\mathbb{R}_{+}^{n+1})$
has been constructed with $\operatorname*{Tr}(\,v_{j})=0$, $\operatorname*{Tr}%
(\,\frac{\partial^{i}v_{j}}{\partial t^{i}})=0$ for all $i=1,\ldots,j-1$,%
\[
\operatorname*{Tr}\left(\,\frac{\partial^{j}v_{j}}{\partial t^{j}}\right)=f_{j}%
-\sum_{i=0}^{j-1}\operatorname*{Tr}\left(  \frac{\partial^{j}v_{i}}{\partial
t^{j}}\right)  ,
\]
and
\[
\Vert\nabla^{m+1}v_{j}\Vert_{L_{k}^{1}(\mathbb{R}_{+}^{n+1})}\lesssim\sum
_{i=0}^{j}|f_{i}|_{B^{m-k-i,1}(\mathbb{R}^{n})}.
\]
By (\ref{m7}), we can apply Step 1, with $f_{j+1}$ replaced by $f_{j+1}%
-\sum_{i=0}^{j}\operatorname*{Tr}\left(  \frac{\partial^{j+1}v_{i}}{\partial
t^{j+1}}\right)  $, to find a function $v_{j+1}\in\dot{W}_{k}^{m+1,1}%
(\mathbb{R}_{+}^{n+1})$ such that $\operatorname*{Tr}(\,v_{j+1})=0$,
$\operatorname*{Tr}(\,\frac{\partial^{i}v_{j+1}}{\partial t^{i}})=0$ for all
$i=1,\ldots,j$, and $\operatorname*{Tr}(\,\frac{\partial^{j+1}v_{j}}{\partial
t^{j+1}})=f_{j+1}-\sum_{i=0}^{j}\operatorname*{Tr}\left(  \frac{\partial
^{j+1}v_{i}}{\partial t^{j+1}}\right)$, with
\begin{align*}
\Vert\nabla^{m+1}v_{j+1}\Vert_{L_{k}^{1}(\mathbb{R}_{+}^{n+1})}  &
\lesssim|f_{j+1}|_{B^{m-k-j-1,1}(\mathbb{R}^{n})}+\sum_{i=0}^{j}\left\vert
\operatorname*{Tr}\left(  \frac{\partial^{j+1}v_{i}}{\partial t^{j+1}}\right)
\right\vert _{B^{m-k-i,1}(\mathbb{R}^{n})}\\
& \lesssim\sum_{i=0}^{j+1}|f_{i}|_{B^{m-k-i,1}(\mathbb{R}^{n})}.
\end{align*}
This induction process gives functions $v_{0}$, \ldots, $v_{l}\in\dot{W}%
_{k}^{m+1,1}(\mathbb{R}_{+}^{n+1})$. Again by (\ref{m7}), we can apply Lemma \ref{lemma normal L1} construct a function $v_{m-k}\in\dot{W}_{k}^{m+1,1}(\mathbb{R}_{+}^{n+1})$
such that $\operatorname*{Tr}(\,v_{m})=0$, $\operatorname*{Tr}(\,\frac
{\partial^{i}v_{m-k}}{\partial t^{i}})=0$ for all $i=1,\ldots,l$,
$\operatorname*{Tr}(\,\frac{\partial^{m-k}v_{m}-k}{\partial t^{m-k}}%
)=f_{m-k}-\sum_{i=0}^{l}\operatorname*{Tr}\left(  \frac{\partial^{m-k}v_{i}%
}{\partial t^{m-k}}\right)  $ with
\begin{align*}
\Vert\nabla^{m+1}v_{m-k}\Vert_{L_{k}^{1}(\mathbb{R}_{+}^{n+1})}  &
\lesssim\Vert f_{m-k}\Vert_{L^{1}(\mathbb{R}^{n})}+\sum_{i=0}^{l}\left\Vert
\operatorname*{Tr}\left(  \frac{\partial^{m-k}v_{i}}{\partial t^{m-k}}\right)
\right\Vert _{L^{1}(\mathbb{R}^{n})}\\
& \lesssim\Vert f_{m-k}\Vert_{L^{1}(\mathbb{R}^{n})}+\sum_{i=0}^{l}%
|f_{i}|_{B^{m-k-i,1}(\mathbb{R}^{n})}.
\end{align*}
We now define $u=v_{0}+\cdots+v_{m}\in\dot{W}_{k}^{m+1,1}(\mathbb{R}_{+}%
^{n+1})$. By construction $\operatorname*{Tr}(\,u)=f_{0}$, $\operatorname*{Tr}%
(\,\frac{\partial^{j}u}{\partial t^{j}})=f_{j}$ for all $j=1,\ldots,m$, and
\[
\Vert\nabla^{m+1}u\Vert_{L_{k}^{1}(\mathbb{R}_{+}^{n+1})}\lesssim\sum_{j=0}%
^{l}|f_{j}|_{B^{m-a-j,1}(\mathbb{R}^{n})}+\Vert f_{m-k}\Vert_{L^{1}%
(\mathbb{R}^{n})}.
\]
To obtain a function in $W^{m+1}_k(\mathbb{R}_{+}^{n+1})$ we proceed as in
Theorem \ref{theorem non-homogeneous}.

The case $a\notin\mathbb{N}_{0}$ is similar but simpler. We omit the details.
\end{proof}

\begin{remark}
Note that when $a=0$ we construct a function $u\in\dot{W}^{m+1,1}%
(\mathbb{R}_{+}^{n+1})$ such that $\operatorname*{Tr}(\,u)=f_{0}$,
$\operatorname*{Tr}(\,\frac{\partial^{j}u}{\partial t^{j}})=f_{j}$ for
$j=1,\ldots,m$, and%
\[
\Vert\nabla^{m+1}u\Vert_{L^{1}(\mathbb{R}_{+}^{n+1})}\lesssim\sum_{j=0}%
^{m-1}|f_{j}|_{B^{m-j,1}(\mathbb{R}^{n})}+\Vert f_{m}\Vert_{L^{1}%
(\mathbb{R}^{n})}.
\]
This estimate was used by Gmeineder, Raita, and Van Schaftingen
\cite{gmeineder-raita-van-schaftingen2022}.\end{remark}

Next, we prove Theorem \ref{theorem a>m}. The proof follows the approach of Grisvard \cite{grisvard1963}, who considered the case $p>1$, $m=0$ and $a\ge p-1$.

\begin{proof}
[Proof of Theorem \ref{theorem a>m}] Let $a>m$ and let $u\in W_{a}^{m+1,1}(\mathbb{R}_{+}^{n+1})$.

\textbf{Step 1:} Assume first that $u\in C^{\infty}(\mathbb{R}_{+}^{n+1})$ with $u=0$ outside
$B_{n}(0,r)\times(0,r)$ for some large $r>0$. Consider a function $\varphi\in
C^{\infty}([0,\infty))$ such that $0\leq\varphi\leq1$, $\varphi=0$ in
$[0,1]$ and $\varphi=1$ in $[2,\infty)$. For $n\in\mathbb{N}$ define
$v_{j}(x,t):=u(x,t)\varphi(jt)$. Given a
multi-index $\alpha=(\beta,l)\in\mathbb{N}_{0}^{n}\times\mathbb{N}_{0}$, with
$|\beta|+l=|\alpha|=m+1$, the product rule implies that
\[
\partial^{\alpha}v_{j}(x,t)=\sum_{i=0}^{l}\binom{l}{i}j^{i}\varphi
^{(i)}(jt)\frac{\partial^{l-i}\partial^{\beta}u}{\partial t^{l-i}\partial
x^{\beta}}(x,t).
\]
If $i=0$, we can use the Lebesgue dominated convergence theorem to show that
\[
\varphi(jt)\partial^{\alpha}u(x,t)\rightarrow\partial^{\alpha}u(x,t)\text{ in
}L_{a}^{1}(\mathbb{R}_{+}^{n+1}).
\]
On the other hand, if $i\geq1$, then by the change of variables $r=jt$, we
have the estimate
\begin{align*}
j^{i}\int_{\mathbb{R}_{+}^{n+1}}t^{a}\left\vert \varphi^{(i)}(jt)\frac
{\partial^{l-i}\partial^{\beta}u}{\partial t^{l-i}\partial x^{\beta}%
}(x,t)\right\vert \,dxdt  & \lesssim\left\Vert \nabla^{l-i+|\beta|}u\right\Vert
_{\infty}\mathcal{L}^{n}(B_n(0,r))j^{i}\int_{1/j}^{2/j}t^{a}dt\\
& \lesssim\frac{1}{j^{a+1-i}}\rightarrow0
\end{align*}
since $a>m$ and $1\leq i\leq m+1$. 

\textbf{Step 2:}  The general case $u\in W_{a}^{m+1,1}%
(\mathbb{R}_{+}^{n+1})$ can be obtained by a density argument. By reflecting  (see, e.g., \cite[Exercise 13.3]{leoni-book-sobolev}) and mollifying $u$, we can assume that 
$u\in W_{a}^{m+1,1}
(\mathbb{R}_{+}^{n+1})\cap C^{\infty}
(\mathbb{R}^{n+1})$. Consider a cut-off function $
\phi\in C^{\infty}_{c}(\mathbb{R}^{n+1})$ such that $\phi=1$ in $B(0,1)$ and $\phi=0$ outside $B(0,2)$. The function $u_j$, given by  $u_j(x,t):=\phi (j^{-1}(x,t))u(x,t)$, satisfies the hypotheses of Step 1 and converges to $u$ in $W^{1,1}_{a}(\mathbb{R}_{+}^{n+1})$ as $j\to\infty$. We omit the details  (see \cite[Lemma 1.2]{grisvard1963} for the case $m=1$).

\end{proof}

\section{Harmonic Extension}\label{section poisson}
The initial goal of this paper was to give a straightforward proof of the estimate for the missing cross terms in \cite{mazya-book2011}, where the following idea emerged.  Following Uspenski\u{\i}, we introduce the harmonic extension of a function $f\in B^{1,1}%
(\mathbb{R}^{n})$:
\begin{equation}
u(x,t):=(P_{t}\ast f)(x)=\int_{\mathbb{R}^{n}}P_{t}%
(x-y)f(y)\,dy,\label{function F harmonic}%
\end{equation}
where $P_{t}$ is the Poisson kernel (cf \cite[p.~61]{stein-book1970})%
\begin{equation}
P(x):=\frac{c_{n}}{(|x|^{2}+1)^{(n+1)/2}},\quad P_{t}(x):=\frac{1}{t^{n}%
}P(xt^{-1})=\frac{c_{n}t}{(|x|^{2}+t^{2})^{(n+1)/2}}\label{kernel harmonic}%
\end{equation}
and%
\begin{equation}
c_{n}=\frac{1}{\int_{\mathbb{R}^{n}}\frac{1}{(|x|^{2}+1)^{(n+1)/2}}dx}%
=\Gamma((n+1)/2)/\pi^{(n+1)/2},\label{constant c_n}%
\end{equation}
where $\Gamma$ is the Gamma function.  

As mentioned in the introduction, Uspenski\u{\i}'s argument \cite{uspenskii1970} on p.~137-138 shows that for $i,j=1,\ldots, n$ one has %
\begin{align}
\frac{\partial^{2}u}{\partial x_{i}\partial x_{j}}(x,t)=\frac{1}{2}%
t^{-n-2}\int_{\mathbb{R}^{n}}\frac{\partial^{2}P_1}{\partial x_{i}\partial
x_{j}}(ht^{-1})[f(x+h)+f(x-h)-2f(x)]\,dh,\label{usp}
\end{align}
which relies on the fact that $\frac{\partial
^{2}P_1}{\partial x_{i}\partial x_j}$ is even and has mean zero.  This is sufficient to estimate the pure second order derivatives in the trace variable.  Meanwhile, harmonicity allows one to reduce the pure second order derivatives in the normal variable to this case, as 
\begin{align*}
\frac{\partial
^{2}u}{\partial t^2}(x,t) = -\sum_{i=1}^n\frac{\partial
^{2}u}{\partial x_i^2}(x,t).
\end{align*}
We then observed that for the mixed case one can simply use the identities 
\begin{equation}
R_{i}\left(  \frac{\partial^{2}P_{t}}{\partial t\partial x_{j}}\right)
=\frac{\partial^{2}P_{t}}{\partial x_{i}\partial x_{j}}, \quad P_{t}\ast f=\sum_{i=1}^{n}R_{i}(P_{t})\ast R_{i}(f),\label{201}%
\end{equation}
where $R_i$ is the Riesz transform, to write 
\begin{align*}
\frac{\partial^{2}u}{\partial t\partial x_{j}}&(x,t)\\&=\frac{1}{2t^2}%
\sum_{i=1}^n\int_{\mathbb{R}^{n}}\frac{\partial^{2}P_t}{\partial x_{i}\partial
x_{j}}(h)[R_i(f)(x+h)+R_i(f)(x-h)-2R_i(f)(x)]\,dh.
\end{align*}
The estimate for the pure second order derivatives can then be applied, using the fact that
for every $f\in B^{1,1}(\mathbb{R}^{n})$,
\begin{equation}
|R_{j}(f)|_{B^{1,1}(\mathbb{R}^{n})}\lesssim|f|_{B^{1,1}(\mathbb{R}^{n})}.
\end{equation}
This estimate is well-known and its classical proof makes use of the
Littlewood--Paley theory (see, e.g., \cite{sawano2020} or \cite[Section
5.2.2]{triebel-book2010}).  We refer to \cite{LeoniSpector2024} for a 
different proof that relies on the intrinsic seminorm of $\dot{B}%
^{1,1}(\mathbb{R}^{n})$ and is based on an argument of Devore,
Riemenschneider, Sharpley \cite{devore-riemenschneider-sharpley1979}.

This argument yields a third proof of the following theorem\footnote{\cite[Theorem 3 on p.~ 135]{uspenskii1970} is accomplished via the harmonic extension, \cite[Theorem 1.9 on p.~356]%
{mironescu-russ-2015} is accomplished via Littlewood--Paley theory, while Burenkov \cite[Theorem 3 in Section 5.4]{burenkov-book1998} gives a different
proof of \eqref{trace space} that covers the case $p=1$ with $a=0$. However, a
crucial point in his proof is the ability to factor the derivative of a
mollifier as a linear combination of another integrable function, which in the
context of the Poisson kernel in the second order case essentially amounts to
showing the existence of an integrable function $\nu$ such that
\[
\frac{\partial^{2}P_{1}}{\partial t \partial x_{j}}(\xi)=2\nu(\xi)-\frac
{1}{2^{n}}\nu(\xi/2).
\]
This is a step we have not been able to verify in the demonstration of
Corollary 7 in Section 5.4 of \cite[Theorem 3 in Section 5.4]%
{burenkov-book1998}.}.
\begin{theorem}
\label{theorem trace}Let $f\in B^{1,1}(\mathbb{R}^{n})$ and let $u$ be defined as in \eqref{function F harmonic}. Then 
\begin{equation}
\Vert \nabla^2 u\Vert_{L^{1}(\mathbb{R}_{+}^{n+1})}\lesssim|f|_{B^{1,1}%
(\mathbb{R}^{n})}. \label{trace estimate}%
\end{equation}

\end{theorem}

After this paper was completed, Mironescu directed us to yet another approach to deal with the cross derivatives, which works under the additional assumption that $u(x,t)\to 0$ as $t\to\infty$ and relies on Taibleson's \cite[Theorem 1 on p.~420]{taibleson1964} (see also \cite[Lemma 4.1 and formula (5.8)]{mironescu-russ-2015}) to estimate the cross term via the pure trace derivatives: 
\begin{align*}
\int_0^\infty \int_{\mathbb{R}^n}  \left|\frac{\partial
^{2}u}{\partial x_{i}\partial t}(x,t) \right|\;dxdt \lesssim    \max_{j=1,\ldots,n} \int_0^\infty \int_{\mathbb{R}^n} \left|\frac{\partial
^{2}u}{\partial x_{i}\partial x_j}(x,y) \right|\;dxdt
\end{align*}
Actually, this is just a concise presentation of the original argument of Uspenski\u{\i}: a combination of Hardy's inequality (\cite[Theorem 1]{uspenskii1970} in his paper or \cite[equation (2.3) in Proposition 2.1 on p.~358]{mironescu-russ-2015} in Mironescu and Russ's), the semi-group property of the Poisson kernel that allows one to express the lifting as a double convolution, and easy estimates of derivatives for the Poisson kernel.  

The following is a more rigorous reiteration of the preceding discussion.  To this end we recall some basic properties of the Riesz transform. Given $j\in\{1,\ldots,n\}$ and a
locally integrable function $f:\mathbb{R}^{n}\rightarrow\mathbb{R}$, the
\emph{Riesz transform} of $f$ is defined formally as%
\begin{equation}
R_{j}(f)(x)=c_{n}\lim_{\varepsilon\rightarrow0^{+}}\int_{\mathbb{R}%
^{n}\setminus B(0,\varepsilon)}f(x-y)\frac{y_{j}}{|y|^{n+1}}%
\,dy,\label{riesz transform}%
\end{equation}
provided the limit exists. The constant $c_{n}$ here is the same as in \eqref{constant c_n}. 
\bigskip

\bigskip

\begin{proposition}
\label{proposition riesz even}Let $P_{t}$ be the Poisson kernel
\eqref{kernel harmonic}. Then%
\[
R_{j}\left(  \frac{\partial P_{t}}{\partial t}\right)  =\frac{\partial P_{t}%
}{\partial x_{j}}.
\]

\end{proposition}

\begin{proof}
Taking the Fourier transform in the $x$ variables gives (see \cite[p.~125]%
{stein-book1970})%
\[
\left(  R_{j}\left(  \frac{\partial P_{t}}{\partial t}\right)  \right)
^{\wedge}(\xi)=i\frac{\xi_{j}}{|\xi|}\left(  \frac{\partial P_{t}}{\partial
t}\right)  ^{\wedge}(\xi)=i\frac{\xi_{j}}{|\xi|}\frac{\partial\hat{P_{t}}%
}{\partial t}(\xi).
\]
As $\hat{P}_{t}(\xi)=e^{-2\pi|\xi|t}$, we have
\[
\frac{\partial\hat{P_{t}}}{\partial t}(\xi)=-2\pi|\xi|e^{-2\pi|\xi|t},
\]
and therefore
\[
\left(  R_{j}\left(  \frac{\partial P_{t}}{\partial t}\right)  \right)
^{\wedge}(\xi)=i\frac{\xi_{j}}{|\xi|}\frac{\partial\hat{P_{t}}}{\partial
t}(\xi)=-2\pi i\xi_{j}e^{-2\pi|\xi|t}=\left(  \frac{\partial P_{t}}{\partial
x_{j}}\right)  ^{\wedge}(\xi).
\]
The claim follows by inverting the Fourier transform.
\end{proof}

\bigskip

\begin{proposition}
\label{proposition product riesz}Let $f\in L^{p}(\mathbb{R}^{n})$ and $g\in
L^{p^{\prime}}(\mathbb{R}^{n})$, where $1<p<\infty$. Then%
\[
\int_{\mathbb{R}^{n}}fg\,dx=\sum_{j=1}^{n}\int_{\mathbb{R}^{n}}R_{j}%
(f)R_{j}(g)\,dx.
\]

\end{proposition}

\begin{proof}
Assume first that $f,g\in\mathcal{S}(\mathbb{R}^{n})$. By making use of
Parseval's identity and the facts that $\mathcal{F}(R_{j}(f))(\xi)=i\frac
{\xi_{j}}{|\xi|}\mathcal{F}(f)(\xi)$ and $\mathcal{F}(f)\in L^{2}%
(\mathbb{R}^{n})$, we have%
\begin{align*}
\int_{\mathbb{R}^{n}}fg\,dx  &  =\int_{\mathbb{R}^{n}}\mathcal{F}%
(f)\overline{\mathcal{F}(g)}\,d\xi=\sum_{j=1}^{n}\int_{\mathbb{R}^{n}}\left(
i\frac{\xi_{j}}{|\xi|}\right)  \mathcal{F}(f) \overline{\left(  i\frac{\xi
_{j}}{|\xi|}\right)  \mathcal{F}(g)}\,d\xi\\
&  =\sum_{j=1}^{n}\int_{\mathbb{R}^{n}}\mathcal{F}(R_{j}(f)) \overline
{\mathcal{F}(R_{j}(g))}\,d\xi=\sum_{j=1}^{n}\int_{\mathbb{R}^{n}}R_{j}%
(f)R_{j}(g)\,dx.
\end{align*}
If $f\in L^{p}(\mathbb{R}^{n})$ and $g\in L^{p^{\prime}}(\mathbb{R}^{n})$, we
use a density argument and the fact that the Riesz transform is bounded in
$L^{q}(\mathbb{R}^{n})$ for all $1<q<\infty$.
\end{proof}

\begin{proof}[First proof of Theorem \ref{theorem trace}] Assume that $f\in C^{\infty
}_c(\mathbb{R}^{n})$ and let $u=P_t*f$, where  $P_{t}$ is the Poisson kernel (\ref{kernel harmonic}).
Since $\left\vert \frac{\partial^{2}P_{1}}{\partial
x_{i}\partial x_{j}}(x)\right\vert \lesssim\frac{1}{|x|^{n+3}}$ for $|x|\geq1$, the estimate \eqref{102}
 holds with $W_t$ replaced by $P_t$. Hence, in view of Remark \ref{remark pure derivatives},  
we can estimate the $L^1$ norm of $\frac{\partial^2u}{\partial x_i\partial x_j}$ as in Step 1 of the proof of Theorem \ref{theorem trace B11}. 
Since
\[
\frac{\partial^{2}P_{t}}{\partial t^{2}}+\sum_{i=1}^{n}\frac{\partial^{2}%
P_{t}}{\partial x_{i}^{2}}=0,
\]
using (\ref{100}), we can write
\[
\frac{\partial^{2}u}{\partial t^{2}}(x,t)=-\sum_{i=1}^{n}\frac{\partial^{2}%
u}{\partial x_{i}^{2}}(x,t)=-\frac{1}{2}\sum_{i=1}^{n}\int_{\mathbb{R}^{n}%
}\frac{\partial^{2}P_{t}}{\partial x_{i}^{2}}(h)[f(x+h)+f(x-h)-2f(x)]\,dh
\]
and in the same way, argue the estimate for this derivative.

\textbf{Step 2: } To estimate the $L^1$ norm of $\frac{\partial^{2}%
u}{\partial t\partial x_{j}}$, we use Taibleson's \cite[Theorem 1 on p.~420]{taibleson1964} applied to the harmonic function $\frac{\partial
^{2}u}{\partial x_{i}\partial t}$:
\begin{align*}
\int_0^\infty \int_{\mathbb{R}^n}  \left|\frac{\partial
^{2}u}{\partial x_{i}\partial t}(x,t) \right|\;dxdt \lesssim   \sum_{j=1}^n \int_0^\infty \int_{\mathbb{R}^n} \left|\frac{\partial
^{2}u}{\partial x_{i}\partial x_j}(x,y) \right|\;dxdt,
\end{align*}
which reduces the argument again the previous case.

\textbf{Step 3: }
A standard density argument in $B^{1,1}(\mathbb{R}^{n})$ allows one to remove the
additional hypothesis that $f\in C^{\infty}_c(\mathbb{R}^{n})$. We omit the details.
\end{proof}

Interestingly, while the boundedness of the Riesz transforms gives a simple proof of the inclusion \eqref{inclusion hard}, i.e. the lifting estimate, the trace characterization of $W^{2,1}(\mathbb{R}^{n+1}_+)$ via harmonic extension itself yields a simple proof that Riesz transforms are bounded on the Besov space $B^{1,1}(\mathbb{R}^{n})$.  In particular, we next establish
\begin{theorem}
\label{theorem riesz}For every $f\in B^{1,1}(\mathbb{R}^{n})$,
\[
|R_{j}(f)|_{B^{1,1}(\mathbb{R}^{n})}\lesssim|f|_{B^{1,1}(\mathbb{R}^{n})}.
\]
\end{theorem}

\begin{remark}
\label{remark well-defined}We observe that if $f\in B^{1,1}(\mathbb{R}^{n})$,
then $f\in W^{1,1}(\mathbb{R}^{n})$ (see \cite[Theorem 17.66]%
{leoni-book-sobolev}). If $n=1$, this implies that $f\in L^{1}(\mathbb{R})\cap
L^{\infty}(\mathbb{R})$, and in turn, $f\in L^{p}(\mathbb{R})$ for all $1\leq
p\leq\infty$. On the other hand, if $n\geq2$, then by the
Sobolev--Gagliardo--Nirenberg embedding theorem, we have $f\in L^{n/(n-1)}%
(\mathbb{R}^{n})$. In both cases, the Riesz transform of $f$ is well-defined.
\end{remark}

Theorem \ref{theorem riesz} is well-known and its classical proof makes use of the
Littlewood--Paley theory (see, e.g., \cite{sawano2020} or \cite[Section
5.2.2]{triebel-book2010}).  The simple proof proceeds as follows.

\begin{proof}[Proof of Theorem \ref{theorem riesz}]
Let $f_\epsilon := f \ast \rho_\epsilon$ for standard mollifiers $\rho_\epsilon$.  Then using the relation (which can be argued using Remark \ref{remark well-defined}, for example)
\begin{align*}
R_{j}(f_\epsilon) \equiv \rho_\epsilon \ast R_{j}(f),
\end{align*}
one observes that $R_{j}(f_\epsilon)$ is a smooth function.  Therefore \cite[Theorem 18.57 on p.~630]{leoni-book-sobolev} gives the inequality
\begin{align*}
    |R_{j}(f_\epsilon)|_{B^{1,1}(\mathbb{R}^{n})} &\lesssim \int_{\mathbb{R}_{+}^{n+1}}|\nabla^{2} P_t\ast R_j(f_\epsilon)| \, dxdt.
\end{align*}
Next observe that for every $s \geq t>0$, $ \nabla^{2} P_s\ast R_j(f_\epsilon)$ is a harmonic function such that
\begin{align*}
\int_{\mathbb{R}^n} |\nabla^{2} P_s\ast R_j(f_\epsilon)| \;dx \leq C_t.
\end{align*}
As a result Stein and Weiss's \cite[Theorem 2.6 on p.~51]{SteinWeiss} gives the bound
\begin{align*}
\|\nabla^{2} P_s\ast R_j(f_\epsilon)\|_{L^\infty(\mathbb{R}^n)} \leq \frac{C'_t}{s^n}.
\end{align*}
Thus one can apply the fundamental theorem of calculus to obtain
\begin{align*}
\nabla^{2} P_t\ast R_j(f_\epsilon) = -\int_t^\infty \frac{\partial }{\partial t} \nabla^{2} P_s\ast R_j(f_\epsilon)\;ds,
\end{align*}
which in combination with Hardy's inequality \cite[equation (2.3) in Proposition 2.1 on p.~358]{mironescu-russ-2015} and Proposition \ref{proposition riesz even} yields 
\begin{align*}
    \int_{\mathbb{R}_{+}^{n+1}}|\nabla^{2} P_t\ast R_j(f_\epsilon)| \, dxdt &\lesssim \int_{\mathbb{R}_{+}^{n+1}} t\left|\nabla^{2} \frac{\partial P_t}{\partial t} \ast R_j(f_\epsilon)\right|\, dxdt\\
     &= \int_{\mathbb{R}_{+}^{n+1}} t\left|\nabla^{2} \frac{\partial P_t}{\partial x_j} \ast f_\epsilon\right|\, dxdt.
\end{align*}
Next Taibleson's \cite[Lemma 4(b) on p.~419]{taibleson1964} and the argument presented in the introduction give
\begin{align*}
\int_{\mathbb{R}_{+}^{n+1}} t\left|\nabla^{2} \frac{\partial P_t}{\partial x_j} \ast f_\epsilon\right| dxdt &\lesssim \int_{\mathbb{R}_{+}^{n+1}} |\nabla^{2} P_t \ast f_\epsilon| \,dxdt \\
&\lesssim  |f_\epsilon|_{B^{1,1}(\mathbb{R}^{n})}.
\end{align*}
These inequalities, the definition of the semi-norm on $B^{1,1}(\mathbb{R}^{n})$, and two change of variables yields
\begin{align*}
    |\rho_\epsilon \ast R_{j}(f)|_{B^{1,1}(\mathbb{R}^{n})} &\lesssim |f_\epsilon|_{B^{1,1}(\mathbb{R}^{n})} \\
    &\leq |f|_{B^{1,1}(\mathbb{R}^{n})},
\end{align*}
so that the claim follows from sending $\epsilon \to 0$ and using Fatou's lemma.
\end{proof}

Conversely, taking for granted that the Riesz transforms are bounded on the Besov spaces, in place of Taibleson's argument in Step 2 in the proof of Theorem \ref{theorem trace}, one has

\begin{proof}[Second proof of Theorem \ref{theorem trace}]

\textbf{Step 2$'$: } To estimate the $L^1$ norm of $\frac{\partial^{2}%
u}{\partial t\partial x_{j}}$, one can alternatively use Proposition \ref{proposition riesz even}. In particular,
by differentiating the equality asserted in the proposition by $x_{j}$, we obtain
\begin{equation}
R_{i}\left(  \frac{\partial^{2}P_{t}}{\partial t\partial x_{j}}\right)
=\frac{\partial^{2}P_{t}}{\partial x_{i}\partial x_{j}}. \label{101}%
\end{equation}
This relation, in combination with Proposition \ref{proposition product riesz}%
, yields the identity
\begin{align*}
\frac{\partial^{2}u}{\partial t\partial x_{j}}(x,t)  &  =\int_{\mathbb{R}^{n}%
}\sum_{i}R_{i}\left(  \frac{\partial^{2}P_{t}}{\partial t\partial x_{j}%
}\right)  (h)R_{i}(f)(x-h)\,dh\\
&  =\int_{\mathbb{R}^{n}}\sum_{i}\frac{\partial^{2}P_{t}}{\partial
x_{i}\partial x_{j}}(h)R_{i}(f)(x-h)\,dh.
\end{align*}
An estimate for this mixed partial derivative of $u$ can therefore be made by
the same argument in Step 1 of the proof of Theorem \ref{theorem trace B11}, which results in the estimate
\begin{align*}
\int_{\mathbb{R}_{+}^{n+1}}  &  \left\vert \frac{\partial^{2}%
u}{\partial t\partial x_{j}}(x,t)\right\vert \,dxdt\\
&  \lesssim\sum_{i}\int_{\mathbb{R}^{n}}\int_{\mathbb{R}^{n}}\frac
{|R_{i}(f)(x+h)+R_{i}(f)(x-h)-2R_{i}(f)(x)|}{|h|^{n+1}}\,dhdx.
\end{align*}
Finally, by Theorem \ref{theorem riesz}, the right-hand side is bounded from
above by $|f|_{B^{1,1}(\mathbb{R}^{n})}$, up to a multiplicative constant.
\end{proof}

The following is the  weighted, higher-order version of Theorem \ref{theorem poisson}:

\begin{theorem}
\label{theorem poisson} Let $m\in\mathbb{N}_{0}$ and $-1<a<m$. Suppose
that $f\in B^{m-a,1}(\mathbb{R}^{n})$ and let $u$ be given by
defined in \eqref{function F harmonic}. Then, one has
\[
\int_{\mathbb{R}_{+}^{n+1}}t^{a}|\nabla^{m+1}u(x,t)|\;dtdx\lesssim
\vert f\vert _{B^{m-a,1}(\mathbb{R}^{n})}.
\]
\end{theorem}

\begin{proof}
Let $s:=m-a>0$ and $f\in  B^{s,1}(\mathbb{R}^{n})$.

\textbf{Step 1: }Assume that $s=k\in\mathbb{N}$ and that $f\in C^\infty_c(\mathbb{R}^n)$. Then for every $\alpha
\in\mathbb{N}_{0}^{n}$ with $|\alpha|=k-1$, we have that $\frac{\partial
^{\alpha}f}{\partial x^{\alpha}}\in C^\infty_c(\mathbb{R}^n)$. We claim that properties of harmonic functions and Taibleson's results allow the reduction to a single estimate which depends on the parity of $m-k+1$:  When $m-k+1$ is even, we show that it suffices to prove that
\begin{equation}
\int_{\mathbb{R}_{+}^{n+1}}t^{a}\left\vert \frac{\partial}{\partial t}%
\frac{\partial^{\beta}u}{\partial x^{\beta}}(x,t)\right\vert \;dtdx\lesssim
|f|_{B^{k,1}(\mathbb{R}^{n})} \label{one y derivative}%
\end{equation}
for any multi-index $\beta\in\mathbb{N}_{0}^{n}$ such that $|\beta|=m$, while when $m-k+1$ is odd we show instead it suffices to prove that
\begin{equation}
\int_{\mathbb{R}_{+}^{n+1}}t^{a}\left\vert \frac{\partial^{\beta^{\prime}}%
u}{\partial x^{\beta^{\prime}}}(x,t)\right\vert \;dtdx\lesssim|f|_{B^{k,1}%
(\mathbb{R}^{n})} \label{no y derivative}%
\end{equation}
for any multi-index $\beta^{\prime}\in\mathbb{N}_{0}^{n}$ such that
$|\beta^{\prime}|=m+1$.

Indeed, any entries of the tensor $\nabla^{m+1}u(x,t)$
has either an odd or even number of derivatives in the normal variable $t$.
Therefore iteration of the relation
\[
\frac{\partial^{2}u}{\partial t^{2}}=-\sum_{i=1}^{n}\frac{\partial^{2}%
u}{\partial x_{i}^{2}}%
\]
reduces the estimate to the case where there are either zero or one derivatives in $t$.  In either case, Taibleson's Theorem \cite[Theorem 1 on p.~420]{taibleson1964} allows to correct the final parity of the number of derivatives in the trace variable:  For $m-k+1$ even, if there is no derivative in $t$ one applies 
\cite[Theorem 1 (a) on p.~420]{taibleson1964} to interchange a derivative in some $x_j$ for a derivative in $t$, or leaves the quantity unchanged if there is one derivative in $t$, which reduces the estimate to the proof of the inequality \eqref{one y derivative};  If $m-k+1$ is odd and there are no derivatives in $t$ one leaves the quantity unchanged, or if there is one derivative in $t$, one applies \cite[Theorem 1 (b) on p.~420]{taibleson1964} to interchange a derivative in some $x_j$ for a derivative in $t$, which reduces the estimate to the proof of the inequality \eqref{no y derivative}.

To show the estimate \eqref{one y derivative} for $m-k+1$ even, let $\beta=\gamma+\delta$ be a
decomposition of the multi-index $\beta$ with $|\gamma|=k-1\geq0$ and
$|\delta|=m-k+1=a+1\geq1$. Then
\begin{equation}
\frac{\partial}{\partial t}\frac{\partial^{\beta}u}{\partial x^{\beta}%
}(x,t)=\left(  \frac{\partial}{\partial t}\frac{\partial^{\delta}P_{t}%
}{\partial x^{\delta}}\ast\frac{\partial^{\gamma}f}{\partial x^{\gamma}%
}\right)  (x).\label{m10}%
\end{equation}
As $|\delta|=m-k+1$ is even, a repetition of the argument in Theorem
\ref{theorem trace} with the even function
\[
\frac{\partial}{\partial t}\frac{\partial^{\delta}P_{t}}{\partial x^{\delta}}%
\]
in place of the mixed second partial derivatives of the Poisson kernel leads
one to the desired bound, using the fact that (see the proof of Lemma \ref{lemma heat integral})
\[
\int_{0}^{\infty}t^{a}\left\vert \frac{\partial}{\partial t}\frac
{\partial^{\delta}P_{t}}{\partial x^{\delta}}(h)\right\vert \;dt\lesssim
\frac{1}{|h|^{n+1}}.
\]

Similarly, for the case \eqref{no y derivative}, let $\beta^{\prime}%
=\gamma^{\prime}+\delta^{\prime}$ be a decomposition of the multi-index
$\beta^{\prime}$ with $|\gamma^{\prime}|=k-1\geq0$ and $|\delta^{\prime
}|=m-k+2=a+2\geq2$,
\begin{equation}
\frac{\partial^{\beta^{\prime}}u}{\partial x^{\beta^{\prime}}}(x,t)=\left(
\frac{\partial^{\delta^{\prime}}P_{t}}{\partial x^{\delta^{\prime}}}\ast
\frac{\partial^{\gamma^{\prime}}f}{\partial x^{\gamma^{\prime}}}\right)
(x).\label{m11}%
\end{equation}
As in this case $|\delta^{\prime}|=m-k+1+1$ is even, the argument is as before, where one uses
\[
\int_{0}^{\infty}t^{a}\left\vert \frac{\partial^{\delta^{\prime}}P_{t}%
}{\partial x^{\delta^{\prime}}}(h)\right\vert \;dt\lesssim\frac{1}{|h|^{n+1}}.
\]
A standard density argument in $B^{s,1}(\mathbb{R}^n)$ allows one to remove the
additional hypothesis that $f\in C^{\infty}_c(\mathbb{R}^{n})$. We omit the details.

\textbf{Step 2: }Assume that $s\notin\mathbb{N}$ and let $k:=\lfloor s\rfloor
$. Then for every $\alpha\in\mathbb{N}_{0}^{n}$ with $|\alpha|=k$, we have
that $\frac{\partial^{\alpha}f}{\partial x^{\alpha}}\in\dot{B}^{s-k,1}%
(\mathbb{R}^{n})$. As in Step 1, to estimate the $L^{1}$ norm of $t^{a}%
\nabla^{m+1}u$, it suffices to prove the estimates \eqref{one y derivative}
and \eqref{no y derivative}. To show the estimate \eqref{one y derivative}, we
use \eqref{m10} but now with $|\gamma|=k\geq0$ and $|\delta|=m-k\geq1$. Since
$\int_{\mathbb{R}^{n}}\frac{\partial}{\partial t}\frac{\partial^{\delta}P_{t}%
}{\partial x^{\delta}}(h)\,dh=0$, we can write%
\[
\left(  \frac{\partial}{\partial t}\frac{\partial^{\delta}P_{t}}{\partial
x^{\delta}}\ast\frac{\partial^{\gamma}f}{\partial x^{\gamma}}\right)
(x)=\int_{\mathbb{R}^{n}}\frac{\partial}{\partial t}\frac{\partial^{\delta
}P_{t}}{\partial x^{\delta}}(h)\Delta_{-h}\frac{\partial^{\alpha}f}{\partial
x^{\alpha}}(x)\,dh.
\]
Multiplying both sides by $t^{a}$ and integrating in $(x,t)$ gives%
\[
\int_{\mathbb{R}_{+}^{n+1}}t^{a}\left\vert \left(  \frac{\partial
}{\partial t}\frac{\partial^{\delta}P_{t}}{\partial x^{\delta}}\ast
\frac{\partial^{\alpha}f}{\partial x^{\alpha}}\right)  (x)\right\vert
\,dxdt\leq\int_{\mathbb{R}^{n}}\int_{\mathbb{R}^{n}}\int_{0}^{\infty}%
t^{a}\left\vert \frac{\partial}{\partial t}\frac{\partial^{\delta}P_{t}%
}{\partial x^{\delta}}(h)\right\vert \left\vert \Delta_{-h}\frac
{\partial^{\alpha}f}{\partial x^{\alpha}}(x)\right\vert \,dtdhdx,
\]
where we used Tonelli's theorem. As in the proof of Lemma \ref{lemma heat integral}, we have
\[
\int_{0}^{\infty}t^{a}\left\vert \frac{\partial}{\partial t}\frac
{\partial^{\delta}P_{t}}{\partial x^{\delta}}(h)\right\vert \,dt\lesssim
\frac{1}{|h|^{n+|\delta|-a}}=\frac{1}{|h|^{n+s-k}}.
\]
In turn,%
\[
\int_{\mathbb{R}_{+}^{n+1}}t^{a}\left\vert \left(  \frac{\partial
}{\partial t}\frac{\partial^{\delta}P_{t}}{\partial x^{\delta}}\ast
\frac{\partial^{\alpha}f}{\partial x^{\alpha}}\right)  (x)\right\vert
dxdt\lesssim\int_{\mathbb{R}^{n}}\int_{\mathbb{R}^{n}}\left\vert \Delta
_{-h}\frac{\partial^{\alpha}f}{\partial x^{\alpha}}(x)\right\vert \frac
{dh}{|h|^{n+s-k}}dx.
\]
Similarly, for the case \eqref{no y derivative}, we use \eqref{m11}, with
$|\gamma^{\prime}|=k\geq0$ and $|\delta^{\prime}|=m+1-k\geq2$, to write%
\[
\left(  \frac{\partial^{\delta^{\prime}}P_{t}}{\partial x^{\delta^{\prime}}%
}\ast\frac{\partial^{\gamma^{\prime}}f}{\partial x^{\gamma^{\prime}}}\right)
(x)=\int_{\mathbb{R}^{n}}\frac{\partial^{\delta^{\prime}}P_{t}}{\partial
x^{\delta^{\prime}}}(h)\Delta_{-h}\frac{\partial^{\gamma^{\prime}}f}{\partial
x^{\gamma^{\prime}}}(x)\,dh.
\]
Since 
\[
\int_{0}^{\infty}t^{a}\left\vert \frac{\partial^{\delta^{\prime}}P_{t}%
}{\partial x^{\delta^{\prime}}}(h)\right\vert \,dt\lesssim\frac
{1}{|h|^{n+|\delta^{\prime}|-a-1}}=\frac{1}{|h|^{n+s-k}},
\]
as before we have that%
\[
\int_{\mathbb{R}_{+}^{n+1}}t^{a}\left\vert \left(  \frac
{\partial^{\delta^{\prime}}P_{t}}{\partial x^{\delta^{\prime}}}\ast
\frac{\partial^{\gamma^{\prime}}f}{\partial x^{\gamma^{\prime}}}\right)
(x)\right\vert dxdt\lesssim\int_{\mathbb{R}^{n}}\int_{\mathbb{R}^{n}}\left\vert
\Delta_{-h}\frac{\partial^{\gamma^{\prime}}f}{\partial x^{\gamma^{\prime}}%
}(x)\right\vert \frac{dh}{|h|^{n+s-k}}dx.
\]
\end{proof}
\begin{remark}
    It is possible to give a second proof of Theorem \ref{theorem poisson}, which makes use of the boundedness of the Riesz transform in $B^{1,1}(\mathbb{R}^n)$, Theorem \ref{theorem riesz}. The idea is similar to the second proof of Theorem \ref{theorem trace}.
\end{remark}

\appendix
\section{Appendix}
In this appendix, we show that when $mp>n$ the homogeneous Sobolev space $\dot{W}^{m,p}(\mathbb{R}^{n})$ is embedded into the homogeneous Besov space (or Zygmund space) $\dot B^{m-n/p,\infty}(\mathbb{R}^{n})$. We also prove that a function in $\dot B^{m-n/p,\infty}(\mathbb{R}^{n})$ has polynomial growth. While the latter result is probably known to experts, we have been unable to find a reference. 

\begin{theorem}\label{theorem growth}
Let $m\in\mathbb{N}$ and $1<p<\infty$ be such that $mp>n$. Then 
\begin{equation}
|u|_{B^{m-n/p,\infty}(\mathbb{R}^{n})}\lesssim\Vert\nabla^{m}_xu\Vert
_{L^{p}(\mathbb{R}^{n})} \label{embedding}%
\end{equation}
for all $u\in\dot{W}^{m,p}(\mathbb{R}^{n})$. Moreover, if $\bar{u}$ is a
representative of $u$, then $\bar{u}$ has polynomial growth.
\end{theorem}

\begin{proof}[First proof]
\textbf{Step 1:} Assume that $u\in C^{\infty}(\mathbb{R}^{n})$. For every
$x,h,y\in\mathbb{R}^{n}$ with $h\neq0$, we use the identity%
\[
\Delta_{h}^{m}u(0)=\sum_{k=1}^{m}(-1)^{m-k}\tbinom{m}{k}\Delta_{(k/m)y}%
^{m}u((m-k)h)-(-1)^{m}\Delta_{h-(k/m)y}^{m}u(ky),
\]
which can be proved using the binomial theorem (see the proof of Lemma 17.22
on p. 549 in \cite{leoni-book-sobolev}). Let $r=|h|$. Averaging in $y$ over
$Q_{r}$ gives%
\begin{align*}
|\Delta_{h}^{m}u(0)| &  \leq\sum_{k=1}^{m}\tbinom{m}{k}\frac{1}{r^{n}}%
\int_{Q_{r}}|\Delta_{(k/m)y}^{m}u((m-k)h)|\,dy\\
&  \quad+\frac{1}{r^{n}}\int_{Q_{r}}|\Delta_{h-(k/m)y}^{m}u(ky)|\,dy:=\sum
_{k=1}^{m}\tbinom{m}{k}\mathcal{A}_{k}+\mathcal{B}.
\end{align*}
By the fundamental theorem of calculus and an induction argument,%
\[
\Delta_{(k/m)y}^{m}u((m-k)h)=\sum_{|\alpha|=m}((k/m)y)^{\alpha}\int%
_{(0,1)^{m}}\frac{\partial^{\alpha}u}{\partial x^\alpha}((m-k)h+(t_{1}+\cdots+t_{m})(k/m)y)\,dt.
\]
Hence,%
\begin{align*}
|\mathcal{A}_{k}| &  \lesssim\frac{1}{r^{n}}\int_{(0,1)^{m}}\int_{Q_{r}}%
|y|^{m}|\nabla^{m}_xu((m-k)h+(t_{1}+\cdots+t_{m})(k/m)y)|\,dydt\\
&  \lesssim\frac{1}{r^{n}}\int_{(0,1)^{m}}\frac{1}{(t_{1}+\cdots+t_{m})^{n+m}}%
\int_{Q_{(t_{1}+\cdots+t_{m})(k/m)r}}|z|^{m}|\nabla^{m}_xu((m-k)h+z)|\,dzdt\\
&  \lesssim\frac{1}{r^{n}}\int_{B_{m}(0,\sqrt{m})}\frac{1}{|t|_{m}^{n+m}}%
\int_{Q_{|t|_{m}kr}}|z|^{m}|\nabla^{m}_xu((m-k)h+z)|\,dzdt
\end{align*}
where we made the change of variables $z=(t_{1}+\cdots+t_{m})(k/m)y$,
$dz=[(t_{1}+\cdots+t_{m})(k/m)^{n}dy$ and used the fact that $\sqrt{m}%
|t|_{m}\leq t_{1}+\cdots+t_{m}\leq m|t|_{m}$. Using Tonelli's theorem on the
right-hand side, we obtain%
\[
|\mathcal{A}_{k}|\lesssim\frac{1}{r^{n}}\int_{Q_{\sqrt{m}kr}}|z|^{m}|\nabla
^{m}_xu((m-k)h+z)|\left(  \int_{E_{z}}\frac{1}{|t|_{m}^{n+m}}\,dt\right)  dz,
\]
where
\[
E_{z}:=\{t\in B_{m}(0,\sqrt{m}):\,|z|<k\sqrt{n}|t|_{m}r\}.
\]
By the change of variables $t=\frac{|z|}{r}\xi$, we have
\[
\int_{E_{z}}\frac{1}{|t|_{m}^{n+m}}\,dt\leq\frac{r^{n}}{|z|^{n}}%
\int_{\mathbb{R}^{m}\setminus B_{m}(0,1/(k\sqrt{n}))}\frac{1}{|\xi|_{m}^{n+m}%
}\,d\xi\lesssim\frac{r^{n}}{|z|^{n}}.
\]
Hence, also, by H\"{o}lder's inequality%
\begin{align*}
|\mathcal{A}_{k}| &  \lesssim\int_{Q_{\sqrt{m}kr}}|z|^{m-n}|\nabla
^{m}_xu((m-k)h+z)|\,dz\\
&  \lesssim\left(  \int_{Q_{\sqrt{m}kr}}|z|^{(m-n)p^{\prime}}dz\right)
^{1/p^{\prime}}\left(  \int_{Q_{\sqrt{m}kr}}|\nabla^{m}_xu((m-k)h+z)|^{p}%
dz\right)  ^{1/p}\\
&  \lesssim r^{m-n/p}\Vert\nabla^{m}_xu\Vert_{L^{p}(Q_{cr})},
\end{align*}
since
\begin{align*}
\int_{\sqrt{m}kr}|z|^{(m-n)p^{\prime}}dz &  \leq\int_{B(0,\sqrt{nm}%
r)}|z|^{(m-n)p^{\prime}}dz=\beta_{n}\int_{0}^{\sqrt{nm}r}r^{n-1+(m-n)p^{\prime
}}dr\\
&  \lesssim r^{n+(m-n)p^{\prime}}%
\end{align*}
and $n+(m-n)p^{\prime}>0$ because\footnote{$N(p-1)+(m-N)p=mp-N>0$} $mp>n$.

The term $\mathcal{B}$ can be treated in a similar way. We omit the details.
In conclusion, we have shown that%
\[
|\Delta_{h}^{m}u(0)|\lesssim|h|^{m-n/p}\Vert\nabla^{m}_xu\Vert_{L^{p}%
(\mathbb{R}^{n})}.
\]
By a translation, it follows that
\[
|\Delta_{h}^{m}u(x)|\lesssim|h|^{m-n/p}\Vert\nabla^{m}_xu\Vert_{L^{p}%
(\mathbb{R}^{n})}%
\]
for all $x,h\in\mathbb{R}^{n}$ with $h\neq0$. This shows that%
\[
|u|_{B^{m-n/p,\infty}(\mathbb{R}^{n})}\lesssim\Vert\nabla^{m}_xu\Vert
_{L^{p}(\mathbb{R}^{n})}.
\]
The additional hypothesis that $u\in C^{\infty}(\mathbb{R}^{n})$ can be
removed using a mollification argument. Finally, observe that if $u\in\dot
{W}^{m,p}(\mathbb{R}^{n})$ and $\varphi\in C_{c}^{\infty}(\mathbb{R}^{n})$,
then $u\varphi\in W^{m,p}(\mathbb{R}^{n})$, and so by \cite[Theorem 12.46, p.
378]{leoni-book-sobolev}, $u\varphi\in L^{\infty}(\mathbb{R}^{n})$. In
particular, $u\in L_{\operatorname*{loc}}^{\infty}(\mathbb{R}^{n})$. Hence,
$u\in\dot{B}^{m-n/p,\infty}(\mathbb{R}^{n})$.

\textbf{Step 2:} It remains to show that a function in $\dot{B}^{m-n/p,\infty
}(\mathbb{R}^{n})$ has polynomial growth. Without loss of generality, we may
assume that $m-n/p\leq1$. Indeed, if $m-n/p>1$, let $\ell\in\mathbb{N}$ be
such that $\ell\leq\max\{i\in\mathbb{N}:~i<m-n/p\}$. Then by \cite[Theorem
17.69, p. 575]{leoni-book-sobolev}, $u$ belongs to $\dot{B}^{m-n/p,\infty
}(\mathbb{R}^{n})$ if and only if for every $\alpha\in\mathbb{N}_{0}^{n}$ with
$|\alpha|=\ell$, the weak derivative $\partial^{\alpha}u$ belongs to $\dot
{B}^{m-n/p-\ell,\infty}(\mathbb{R}^{n})$. Moreover, if $\nabla^{\ell}u$ has
polynomial growth, then so does $u$. Thus, in what follows we assume that
$0<m-n/p\leq1$.

If $0<m-n/p<1$, then $u$ has a representative $\bar{u}$ that is H\"{o}lder
continuous with exponent $m-n/p$. In particular, $\bar{u}$ has polynomial growth.

Assume next that $m-n/p=1$. Let $u\in\dot{B}^{1,\infty}(\mathbb{R}^{n})$ and
let $\bar{u}$ be a representative of $u$. Set%
\[
v(h):=\bar{u}(h)-\bar{u}(0),\quad h\in\mathbb{R}^{n}.
\]
Then for every $i\in\mathbb{Z}$ and every $h\in\mathbb{R}^{n}$,
\[
|2^{i}v(2^{-i}h)-2^{i-1}v(2^{-i+1}h)|=2^{i-1}|\Delta_{2^{-i}h}\bar{u}%
(0)|\leq|u|_{B^{1,\infty}(\mathbb{R}^{n})}|h|.
\]
It follows that for every $\ell\in\mathbb{Z}\setminus\{0\}$,
\[
|v(h)-2^{\ell}v(2^{-\ell}h)|\leq\sum_{i=1}^{\ell}|2^{i}v(2^{-i}h)-2^{i-1}%
v(2^{-i+1}h)|\leq|u|_{B^{1,\infty}(\mathbb{R}^{n})}\ell|h|
\]
if $\ell>0$, while
\[
|v(h)-2^{\ell}v(2^{-\ell}h)|\leq\sum_{i=\ell-1}^{0}|2^{i}v(2^{-i}%
h)-2^{i-1}v(2^{-i+1}h)|\leq|u|_{B^{1,\infty}(\mathbb{R}^{n})}|\ell||h|
\]
if $\ell<0$. Hence,
\begin{align*}
|v(2^{-\ell}h)|  &  \leq|v(h)|2^{-\ell}+|v(h)-2^{\ell}v(2^{-\ell}h)|2^{-\ell
}\\
&  \leq|v(h)|2^{-\ell}+|u|_{B^{1,\infty}(\mathbb{R}^{n})}|\ell|2^{-\ell}|h|.
\end{align*}
Given $x\in\mathbb{R}^{n}\setminus\{0\}$, let $\ell\in\mathbb{Z}$ be such that
$2^{-\ell-1}\leq|x|<2^{-\ell}$. Taking $h=2^{\ell}x$ gives%
\begin{equation}
|v(x)|\leq2|v(2^{\ell}x)||x|+\frac{1}{\log2}|u|_{B^{1,\infty}(\mathbb{R}^{n}%
)}|x|(1+|\log|x||). \label{ttt}%
\end{equation}
Hence,
\begin{align*}
|\bar{u}(x)|  &  \leq|\bar{u}(0)|+|v(x)|\\
&  \leq|\bar{u}(0)|+\sup_{B(0,1)}|\bar{u}-\bar{u}(0)||x|+\frac{1}{\log
2}|u|_{B^{1,\infty}(\mathbb{R}^{n})}|x|(1+|\log|x||),
\end{align*}
where we used the fact that $|2^{\ell}x|<1$.
\end{proof}

\begin{remark}
We refer to Peetre \cite[Theorem 8.2]{peetre1966} for the original proof of
Step 1, which relies on interpolation theory and on an identity of the type \eqref{representation} below. Step 2 is adapted from a paper of Krantz
\cite[Lemma 2.8]{krantz1983}. Note that by \eqref{ttt} and a translation,
\[
|\bar{u}(x_{0}+h)-\bar{u}(x_{0})|\leq\sup_{B(x_{0},1)}|\bar{u}-\bar{u}%
(x_{0})||h|+\frac{1}{\log2}|u|_{B^{1,\infty}(\mathbb{R}^{n})}|h|(1+|\log|h||).
\]

\end{remark}

It is possible to provide a shorter proof of the fact that a function in $u\in \dot W^{m,p}(\mathbb{R}^n)$ has polynomial growth. This proof uses a representation result in Mizuta's book \cite[Theorem 1.3 on p.~207]{mizuta-book-1996}, which makes use of the theory of singular integrals.

\begin{proof}[Second proof] 
By \cite[Theorem 1.3 on p.~207]{mizuta-book-1996}, for $v \in \dot W^{m,p}(\mathbb{R}^n)$ there exists a polynomial $P$ of degree at most $m-1$ such that one has for almost every $x \in \mathbb{R}^n$,
\begin{align}\label{representation}
v(x)-P(x) = \sum_{|\lambda| = m} a_\lambda \int k_{\lambda,l}(x,y) \frac{\partial^\lambda v}{\partial y^\lambda}(y)\;dy,
\end{align}
where  $l\leq m-n/p <l+1$  and 
\[
k_{\lambda,l}(x,y):=\left\{
\begin{array}
[c]{ll}%
k_{\lambda}(x-y) & \text{if }|y|<1,\\
k_{\lambda}(x-y)-\sum_{|\alpha|\leq l}\frac{x^{\alpha}}{\alpha!}%
\frac{\partial^{\alpha}k_{\lambda}}{\partial y^\alpha}(-y) & \text{if }|y|\geq1,
\end{array}
\right.%
\]
with $k_{\lambda}(x):=\frac{x^{\lambda}}{|x|^{n}}$. By \cite[Lemma 1.2 on p.~207]{mizuta-book-1996},
\begin{align}
|k_{\lambda,l}(x,y)| \lesssim |x|^{l+1} |y|^{|\lambda|-n-l-1}\label{estimate kernel}
\end{align}
for $|y|\geq 1$ and $|y|>2|x|$.

We claim that 
\begin{align*}
|v(x)-P(x)| \leq |\tilde{P}(x)|
\end{align*}
for some polynomial $\tilde{P}$.  
However, this is between the lines of Mizuta's argument \cite[Proof of Theorem 1.3 on p.~207]{mizuta-book-1996}.  If $|x|\leq 1/2$, one splits the integral into two pieces, uses the bound for $k_\lambda(x-y)$ for the local piece and (\ref{estimate kernel}) for the global piece, and H\"older's inequality to obtain
\begin{align*}
|v(x)-P(x)| &\lesssim \int_{B(0,1)}  |k_{\lambda,l}(x,y)| |\nabla^{m}_y v(y)|\;dy + \int_{\mathbb{R}^n\setminus B(0,1)}  |k_{\lambda,l}(x,y)| |\nabla^{m}_y v(y)|\;dy\\
& \lesssim \int_{B(0,1)}  |x-y|^{|\lambda|-n} |\nabla^{m}_y v(y)|\;dy +  |x|^{l+1} \int_{\mathbb{R}^n\setminus B(0,1)}  |y|^{|\lambda|-n-l-1} |\nabla^{m}_y v(y)|\;dy\\
& \lesssim |P_1(x)|.
\end{align*}
Here one uses that 
\begin{align*}
(|\lambda|-n-l-1)p' +n &=(m-n-l-1)p' +n<0
\end{align*}
because $m-n/p <l+1$.  

When $|x|>1/2$ we use a similar splitting, along with H\"older's inequality
\begin{align*}
|v(x)-P(x)| &\lesssim \int_{B(0, 2|x|)}  |k_{\lambda,l}(x,y)| |\nabla^{m}_y v(y)|\;dy + \int_{\mathbb{R}^n\setminus B(0, 2|x|)}  |k_{\lambda,l}(x,y)| |\nabla^{m}_y v(y)|\;dy\\
&\lesssim \int_{B(0, 2|x|)}  |k_{\lambda}(x-y)| |\nabla^{m}_y v(y)|\;dy +
 | P_2(x)|
\\
& \quad +|x|^{l+1}\int_{\mathbb{R}^n\setminus B(0, 2|x|)}  |y|^{|\lambda|-n-l-1} |\nabla^{m}_y v(y)|\;dy\lesssim |P_3(x)|,
\end{align*}
where we have also made use of Mizuta's observation that for $f\in L^p(\mathbb{R}^n)$,
\begin{align*}
 \int_{B(0, 2|x|)}  k_{\lambda,l}(x,y) f(y)\;dy =  \int_{B(0, 2|x|)}  k_{\lambda}(x-y) f(y)\;dy + \text{ a polynomial.}
\end{align*}
\end{proof}

\section*{Acknowledgements}
The research of G. Leoni was supported by the National Science Foundation
under grant No. DMS-2108784. G. Leoni started working on this problem more
than twenty years ago with Giuseppe Savar\'e and then abandoned it when they
discovered the references \cite[Theorem 1 in Section 10.1]{mazya-book2011} and
\cite[Theorem 3 in Section 5.4]{burenkov-book1998}. He resumed working on it
in 2016 with Irene Fonseca and Maria Giovanna Mora, after realizing the issues
in \cite[Theorem 1 in Section 10.1]{mazya-book2011} and \cite[Theorem 3 in
Section 5.4]{burenkov-book1998}, and abandoned it again after discovering
\cite{mironescu-russ-2015}. D. Spector is supported by the National Science and Technology Council of Taiwan under research grant numbers 110-2115-M-003-020-MY3/113-2115-M-003-017-MY3 and the Taiwan Ministry of Education under the Yushan Fellow Program.  
 The authors would like to thank Bogdan Rai\cb t\v a and Ian Tice for useful conversations on the subject of this paper.  

The authors would like to thank Petru Mironescu for sharing his deep insight into the problem of traces and liftings, which has greatly helped both the presentation and accuracy of our paper.  
%
\bibliographystyle{abbrv}
\bibliography{traces-ref}

\end{document}